\documentclass[12pt]{amsart}

\usepackage{amsfonts,amsmath,latexsym,amssymb,verbatim,amsbsy,times}
\usepackage{amsthm}
\usepackage{pstricks}
\usepackage{graphicx}
\usepackage{caption}
\usepackage[top=0.8in, bottom=0.8in, left=0.8in, right=0.8in]{geometry}

\usepackage[colorlinks=true, pdfstartview=FitV, linkcolor=blue,citecolor=green, urlcolor=blue]{hyperref}

\usepackage{enumitem}

\theoremstyle{plain}
\newtheorem{THEOREM}{Theorem}[section]
\newtheorem{COROL}[THEOREM]{Corollary}
\newtheorem{LEMMA}[THEOREM]{Lemma}
\newtheorem{PROP}[THEOREM]{Proposition}

\newtheorem{lemma}[THEOREM]{Lemma}

\theoremstyle{definition}
\newtheorem{DEF}[THEOREM]{Definition}

\theoremstyle{remark}
\newtheorem{REMARK}[THEOREM]{Remark}
\newtheorem{CLAIM}[THEOREM]{Claim}

\newcommand{\N}{\ensuremath{\mathbb{N}}}   
\newcommand{\Z}{\ensuremath{\mathbb{Z}}}   
\newcommand{\R}{\ensuremath{\mathbb{R}}}   
\newcommand{\T}{\ensuremath{\mathbb{T}}}   

\def \a {\alpha}

\def \d {\delta}
\def \g {\gamma}
\def \e {\epsilon}
\def \f {\varphi}

\def \k {\kappa}
\def \l {\lambda}
\def \L {\Lambda}

\def \s {\sigma}

\def \t {\tau}

\def\cF {\mathcal {F}}

\def \cL {\mathcal{L}}
\def \cM {\mathcal{M}}

\def\cT {\mathcal {T}}

\def \loc {\mathrm{loc}}

\def \< {\langle}
\def \> {\rangle}
\def \p {\partial}
\def \ra {\rightarrow}

\newcommand{\esssup}{\mathop{\mathrm{ess\,sup}}}

\DeclareMathOperator{\supp}{supp} %

\def \dt  {\, \mbox{d}t}

\def \dx  {\, \mbox{d}x}
\def \dy  {\, \mbox{d}y}
\def \dz  {\, \mbox{d}z}

\def \ds  {\, \mbox{d}s}

\begin{document}

	
\title{Weak and Strong Solutions to the Forced Fractional Euler Alignment System}

\author{Trevor M. Leslie}

\email{tlesli2@uic.edu}

\address{Department of Mathematics, Statistics, and Computer Science \\851 S Morgan St, M/C 249 \\ University of Illinois at Chicago, Chicago, IL, 60607}

\begin{abstract}
We consider a hydrodynamic model of self-organized evolution of agents, with singular interaction kernel $\phi_\a(x)=1/|x|^{1+\alpha}$ ($0<\alpha<2$), in the presence of an additional external force.  Well-posedness results are already available for the unforced system in classical regularity spaces.  We define a notion of solution in larger function spaces, in   particular in $L^\infty$ (``weak solutions'') and in $W^{1,\infty}$ (``strong solutions''), and we discuss existence and uniqueness of these solutions. Furthermore, we show that several important properties of classical solutions carry over to these less regular ones.  In particular, we give Onsager-type criteria for the validity of the natural energy law for weak solutions of the system, and we show that fast alignment (weak and strong solutions) and flocking (strong solutions) still occur in the forceless case. 
\end{abstract}
	
\maketitle
	
\section{Introduction}

\subsection{The Forced Euler-Alignment System}	
For some fixed $\a\in (0,2)$, we consider the system 
\begin{equation}
\label{e:mainv}
u_t + uu' = -\L_\a(\rho u) + u\L_\a\rho + f, 
\end{equation}
\begin{equation}
\label{e:maind}
\rho_t + (\rho u)' = 0,
\end{equation}
for $(x,t)\in \T\times [0,\infty)$.  Here and below, we use primes $'$ to denote spatial derivatives.  The torus $\T$ may have arbitrarily large period, but we work on the $2\pi$-periodic torus for the sake of definiteness.  Here $u=u(x,t)$ is the macroscopic velocity, $\rho=\rho(x,t)$ is the density (assumed nonnegative), and $f=f(x,t)$ is an external forcing term, assumed given.  The operator $-\L_\a$ is (up to a constant) the classical fractional Laplacian, with kernel 
\begin{equation}
\label{e:kernel}
\phi_\a(z) = \sum_{k\in \Z} \frac{1}{|z+2\pi k|^{1+\a}},
\quad \quad 
z\in [-\pi,\pi]\backslash \{0\}.
\end{equation}
The action of $-\L_\a$ on a sufficiently regular function $g:\T\to \R$ is given explicitly by 
\[
-\L_\a g(x) = \int_\T (g(x+z) - g(x))\phi_\a(z)\dz = \int_\R (u(x+z) - u(x))\frac{\dz}{|z|^{1+\a}},
\]
with the integral taken in the principal value sense.  Let us temporarily consider the situation where $f\equiv 0$.  In this setting, \eqref{e:mainv}--\eqref{e:maind} becomes a special case of the system 
\begin{equation}
\label{e:genv}
u_t + uu' = \cL_\phi(\rho u) - u\cL_\phi\rho, 
\end{equation}
\begin{equation}
\label{e:gend}
\rho_t + (\rho u)' = 0,
\end{equation}
where $\cL_\phi$ is given by 
\[
\cL_\phi g(x) = \int_\T \phi(|x-y|)(g(y)-g(x))\dy.
\]
The system \eqref{e:genv}--\eqref{e:gend} can in turn be interpreted as a macroscopic limit of the system
\begin{equation}
\label{e:CS}
\left\{
\begin{array}{lcl}
\dot{x}_i & = & v_i, \\
\dot{v}_i & = & \frac{1}{N} \sum_{j=1}^N \phi(|x_i-x_j|)(v_j-v_i),
\end{array}
\right.
\end{equation}
as $N\to \infty$.  The system \eqref{e:CS} is the celebrated Cucker-Smale model \cite{CS2007a}, which describes the positions $x_i$ and velocities $v_i$ of $N$ agents whose binary interaction law depends on the radial influence function $\phi\ge 0$.  We do not attempt an overview of the existing literature related to this model; rather we cite only a few results which are pertinent to the present context and refer the reader to, for example, \cite{CCP2017} and references therein for a more substantial review.  See also the Introduction of \cite{DKRT} for a useful and concise overview of some relevant results.  

The system \eqref{e:CS} and its long-time dynamics are associated with two especially notable phenomena.  First, the velocities align to a constant (given by momentum divided by mass---both of these are conserved), and second, the system exhibits the so-called flocking phenomenon, whereby the agents gather into a crowd of finite diameter.  However, it seems that in order for these characteristics to emerge, the kernel $\phi$ must involve some (non-physical) long-range interactions (c.f. \cite{CCTT2016}, \cite{CS2007a}, \cite{STII}, \cite{TT2014}).  In order to emphasize rather the local interactions, one recent strategy has been to use a kernel $\phi$ which is singular at the origin, for example $\phi = \phi_\a$.  This is the case we treat in the present paper, at the macroscopic level of the system \eqref{e:mainv}--\eqref{e:maind}.  See also \cite{MT2011} for another approach on the level of the agents.

Within the last few years, the system \eqref{e:mainv}--\eqref{e:maind} has received a fair amount of attention; the papers \cite{STI}, \cite{STII}, \cite{STIII}, and \cite{DKRT} all give well-posedness results in classical regularity spaces in the case $f\equiv 0$.  The second and third of these also show that classical solutions of the system exhibit flocking (see below for more details).  In \cite{KT}, the well-posedness of \eqref{e:mainv}--\eqref{e:maind} is studied in the case where $f$ is replaced by an  attraction-repulsion interaction that depends on $\rho$ and (the derivative of) a given kernel $K$.  

In dimensions higher than $1$, there are very few results on the analogues of the systems \eqref{e:mainv}--\eqref{e:maind} or \eqref{e:genv}--\eqref{e:gend}.  He and Tadmor \cite{HT2016} have considered the analogue of \eqref{e:genv}--\eqref{e:gend} in two dimensions in the case of smooth kernels $\phi$.  And very recently, Shvydkoy \cite{S2018} gave the first results to treat the analogue of the (forceless) system \eqref{e:mainv}--\eqref{e:maind} in arbitrary dimensions $n>1$.  The latter work proves a small data result for the range $\a\in (2/3,\;3/2)$.  

The present work differs from all those cited above in that it treats well-posedness in low-regularity spaces, for an arbitrary external force $f$ (which is sufficiently regular).  Before giving more details on the results contained in this paper and past work on the equations, however, we pause to give some definitions that will be helpful in this discussion.

\subsection{Auxiliary Quantities and Notation}
An interesting feature of the system \eqref{e:mainv}--\eqref{e:maind} is that certain combinations of $u$ and $\rho$ formally satisfy conservation laws or transport equations.  For example, define $e:=u' - \L_\a \rho$.  Then the velocity equation can be rewritten as 
\begin{equation}
\label{e:ue}
u_t + ue = -\L_\a(\rho u) + f.
\end{equation}
Differentiating this, applying $\L_\a$ to the density equation, and subtracting, we obtain an evolution equation for $e$:
\begin{equation}
\label{e:econs}
e_t + (ue)' = f'.
\end{equation}
Next, we define $q:=e/\rho$.  Taking the time derivative of $q$ and using the density equation, we see that $q$ satisfies
\begin{equation}
\label{e:qtrans}
q_t + uq' = \frac{f'}{\rho}.
\end{equation}
But then $q'$ satisfies an equation like \eqref{e:econs}:
\begin{equation}
\label{e:qxcons}
q'_t + (uq')' = (q_t + uq')' = \left( \frac{f'}{\rho} \right)'.
\end{equation}
And finally, $q'/\rho$ satisfies an equation like \eqref{e:qtrans}:
\begin{equation}
\label{e:qxrhotrans}
\left( \frac{q'}{\rho} \right)_t + u \left( \frac{q'}{\rho} \right)' = \frac{1}{\rho} \left( \frac{f'}{\rho} \right)'.
\end{equation}
Obviously this process can be continued, but $q'/\rho$ is the highest order quantity of this type that we make use of below.

We also set notation for the (conserved) mass $\cM$ associated to the system:
\[
\cM = \int_\T \rho\dx.
\]

\subsection{Weak, Strong, and Regular Solutions}
We now define several notions of a solution to \eqref{e:mainv}--\eqref{e:maind}.  For weaker notions of a solution, we include $e$ as part of our definitions.  To write down a weak formulation, it is helpful to use \eqref{e:ue} instead of the original velocity equation. We also include a weak form of the definition of $e$.  
\begin{DEF}
Let $(u_0, \rho_0, e_0)\in L^\infty \times L^\infty \times L^\infty$ satisfy the compatibility condition 
\begin{equation}
\label{e:compat}
\int_\T e_0 \f + u_0 \f' + \rho_0 \L_\a \f \dx = 0, 
\quad \text{ for all } \f \in C^\infty(\T)
\end{equation}
We say that $(u,\rho,e)$ is a \emph{weak solution} on the time interval $[0,T]$, satisfying the initial data $(u_0, \rho_0, e_0)$, if 
\begin{itemize}
\item $u$, $\rho$, $\rho^{-1}$, and $e$ all belong to $L^\infty(0,T;L^\infty)$.
\item $u$ and $\rho$ belong to $L^2(0,T;H^{\a/2})$. 
\item $(u,\rho,e)$ satisfies the following weak form of \eqref{e:mainv}--\eqref{e:maind}, for all $\f\in C^\infty(\T\times [0,T])$ and a.e. $t\in [0,T]$:
\begin{equation}
\label{e:weakv}
\int_\T u(t)\f(t)\dx - \int_\T u_0 \f(0)\dx - \int_0^t \int_\T u\p_t \f \dx \ds
= \int_0^t \int_\T -ue\f - \rho u \L_\a\f + f\f\dx\ds, 
\end{equation}
\begin{equation}
\label{e:weakd}
\int_\T \rho(t)\f(t)\dx - \int_\T \rho_0\f(0)\dx - \int_0^t \int_\T \rho \p_t \f \dx \ds
 = \int_0^t \int_\T \rho u \f' \dx\ds.
\end{equation}
\item The compatibility condition \eqref{e:compat} propagates in time, in the sense that 
\begin{equation}
\label{e:compat2}
\int_0^T \int_\T e \f + u \f' + \rho \L_\a\f  \dx \ds 
 = 0, \quad \text{ for all } \f \in C^\infty(\T\times [0,T]).
\end{equation}
\end{itemize}
We say that $(u,\rho,e)$ is a weak solution on $[0,T)$ ($0<T\le \infty$) if $(u, \rho, e)$ is a weak solution on $[0,T']$ for all $T'\in (0,T)$.  
\end{DEF}

\begin{DEF}
Let $(u_0, \rho_0, e_0)\in W^{1,\infty}\times W^{1,\infty} \times W^{1,\infty}$ satisfy the compatibility condition \eqref{e:compat}.  We say that $(u, \rho, e)$ is a \emph{strong solution} on the time interval $[0,T]$, satisfying the initial data $(u_0, \rho_0, e_0)$, if $(u,\rho, e)$ is a weak solution such that $u$, $\rho$, and $e$ all belong to $L^\infty(0,T;W^{1,\infty})$.  We say that $(u,\rho,e)$ is a strong solution on $[0,T)$ ($0<T\le \infty$) if $(u, \rho, e)$ is a strong solution on $[0,T']$ for all $T'\in (0,T)$. 
\end{DEF}

The quantity $e$ need not play a role in the definition of higher-regularity solutions, though it does remain an important quantity for the analysis of such solutions.  
\begin{DEF}
We say that $(u, \rho)$ is a \emph{regular solution} of \eqref{e:mainv}--\eqref{e:maind}, on the time interval $[0,T]$, satisfying the initial condition $(u_0, \rho_0)\in H^4\times H^{3+\a}$, if 
\begin{itemize}
\item $(u, \rho)$ satisfies \eqref{e:mainv}--\eqref{e:maind} in the classical sense.
\item $(u,\rho)\in C([0,T];H^4\times H^{3+\a})$,
\item $\rho(x,t)\ge c$ for some $c>0$, for all $(x,t)\in \T\times [0,T]$, and 
\item $u(0) = u_0$ and $\rho(0)= \rho_0$ a.e. in $\T$.
\end{itemize}
We say that $(u, \rho)$ is a regular solution on the time interval $[0,T)$ ($0<T\le \infty$) if $(u,\rho)$ is a regular solution on $[0,T']$ for all $T'\in (0,T)$.  
\end{DEF}

\subsection{Alignment and Flocking}

In the discussion above, we have already mentioned the phenomena of alignment and flocking in the context of agent-based models.  We now give the more precise definitions associated to the macroscopic system \eqref{e:mainv}--\eqref{e:maind}.
\begin{DEF}
A solution $(u, \rho, e)$ is said to experience \emph{alignment} if the diameter of the velocities tends to zero as $t\to \infty$:
\[
A(t):=\esssup_{x,y\in \T} |u(x,t)-u(y,t)|\to 0, 
\quad
\text{ as } t\to \infty.
\]
We say that the solution undergoes \emph{fast alignment} if the convergence $A(t)\to 0$ is exponentially fast.  
\end{DEF}

\begin{DEF}
We define the set of \emph{flocking states} $\cF$ as follows: 
\[
\cF:=\{(\overline{u}, \overline{\rho}): \overline{u}\equiv \text{constant},\;\overline{\rho}(x,t) = \rho_\infty(x - t\overline{u})\}.
\]
And we say that $(u,\rho)$ converges to a flocking state $(\overline{u}, \overline{\rho})\in \cF$ in the space $X\times Y$ if 
\begin{equation}
\label{e:flockdefn}
\|u(\cdot, t) - \overline{u}\|_X + \|\rho(\cdot, t) - \overline{\rho}(\cdot,t)\|_Y\to 0
\quad 
\text{ as } t\to \infty.
\end{equation}
And we say that $(u, \rho)$ experiences \emph{fast flocking} in $X\times Y$ if the convergence rate of \eqref{e:flockdefn} is exponentially fast.  
\end{DEF}

\subsection{Previous and Present Results}
The existing well-posedness theory for the system \eqref{e:mainv}--\eqref{e:maind} mostly concerns the special case $f\equiv 0$.  In \cite{STI}, a priori estimates implying the local existence of regular solutions for the case $1\le \a<2$ and $f\equiv 0$ were given.  With local existence in hand, the authors refined these estimates and proved global in time existence of solutions.  Further refinements were given in \cite{STII}, which proves that regular solutions undergo fast alignment and converge exponentially quickly to a flocking state in $H^3\times H^{3-\e}$, for any $\e>0$. 

The first treatment of well-posedness for the case $0<\a<1$, $f\equiv 0$ appeared in \cite{DKRT}. Later, the results of \cite{STI}, \cite{STII} were extended to the case $0<\a<1$ in \cite{STIII}, which obtains local and global existence as a byproduct of the proof of flocking and fast alignment.  In \cite{KT} the authors prove results analogous to those of \cite{DKRT}, in the presence of an additional force of the form $-\p_x K * \rho$.  

The techniques used in the two groups of papers \cite{STI}, \cite{STII}, \cite{STIII} and \cite{DKRT}, \cite{KT} are quite different from one another.  The first group uses regularity theory for fractional parabolic equations and relies extensively on the nonlinear maximum principle of \cite{ConstVicol}; the nonlinear maximum principle was originally used to prove a well-posedness result for the critical SQG equation.  The papers \cite{DKRT}, \cite{KT} use instead the modulus of continuity method, which has also been used to treat (for example) the SQG equation in \cite{KNV2007}.

In this paper, we consider the case of a general external force $f$ which is regular enough for our computations to go through. In principle, we could include the force considered in \cite{KT} in our existence results, but to do so we would need to repeat several of the arguments from \cite{KT}, rendering the inclusion somewhat artificial.  The problem is that the density in \cite{KT} is not obviously bounded a priori, and in fact may grow exponentially in time.  Since their force in turn depends on the density, we would need to make quite a few adjustments to our arguments (and intermediate conclusions) in order to include this case.  To simplify our arguments, we assume that our force $f$ and a sufficient number of its spatial derivatives are uniformly bounded in $\T\times [0,\infty)$.  In particular, our arguments ultimately do not apply to the force considered in \cite{KT}.
Rather, we extend the result of \cite{STI}, \cite{STII}, and \cite{STIII}, concerning existence of regular solutions, to the forced case (for nice enough $f$).  We construct both weak and strong solutions as limits of regular solutions.  These solutions are slightly more regular (in the H\"older sense) than one can conclude a priori using only the definitions of weak and strong solutions.  However, the strong solutions we construct are in fact unique within their class; therefore, the regularity properties obtained by the method of construction are enjoyed by all strong solutions.  

The results described thus far are all basically in the spirit of \cite{STI}, \cite{STII}, \cite{STIII} (and, to a certain extent, \cite{KT}).  We also include, however, a discussion of the natural energy laws of the system \eqref{e:mainv}--\eqref{e:maind} which has no counterpart in any of the aforementioned papers.  (The energy equalities are obvious for solutions in classical regularity spaces, so there was no need for such a discussion in those contexts.)  We propose Onsager-type criteria that guarantee that these energy laws hold for weak solutions.  We emphasize that these criteria are valid for \textit{any} weak solutions, not just the ones we construct as limits of regular solutions.  To treat the nonlinear term, we rely on the techniques of \cite{LS2016} (which in turn relies on \cite{CCFS}, \cite{CET}).  However, existing commutator estimates seem to be insufficient to treat the dissipation term directly, and we therefore devote a fair amount of effort to showing that the dissipation cannot cause any problems.  It turns out that our Onsager-type criteria are satisfied for all weak solutions in the case where $\a\in [1,2)$.  For smaller $\a$ one can prove the analogous energy \textit{in}equalities for the constructed solutions, even if our Onsager-type criteria are not satisfied.  

We state our main results in the following four theorems.

\begin{THEOREM}
\label{t:Reg}
Let $(u_0, \rho_0)\in H^4\times H^{3+\a}$, with $\rho_0^{-1}\in L^\infty$, and assume that $f\in L^\infty(0,\infty;H^4)$.  Then there exists a global-in-time regular solution of \eqref{e:mainv}--\eqref{e:maind} associated to the initial data $(u_0, \rho_0)$.  
\end{THEOREM}

\begin{THEOREM}
\label{t:wk}
Let $(u_0, \rho_0, e_0)\in L^\infty\times L^\infty \times L^\infty$ satisfy the compatibility condition \eqref{e:compat}. Assume additionally that $\rho_0^{-1}\in L^\infty$ and that $f\in L^\infty(0,\infty;W^{1,\infty})$.  Then there exists a global-in-time weak solution $(u, \rho, e)$ associated to the initial data $(u_0, \rho_0, e_0)$, which satisfies the following energy inequalities:
\begin{equation}
\label{e:ei}
\frac12 \int_\T \rho u^2(t)\dx + \frac12 \int_0^t \int_\T \int_\R \rho(x)\rho(y) \frac{|u(x)-u(y)|^2}{|x-y|^{1+\a}}\dy\dx\ds \le \frac12 \int_\T \rho_0 u_0^2\dx + \int_0^t \int_\T \rho u f \dx\ds
\end{equation}
\begin{equation}
\label{e:rhoei}
\int_\T \rho(t)^2\dx + \frac12\int_0^t \int_\T \int_\R (\rho(x)+\rho(y)) \frac{|\rho(x)-\rho(y)|^2}{|x-y|^{1+\a}}\dy\dx\ds \le \int_\T \rho_0^2\dx - \int_0^t \int_\T e\rho^2\dx \ds.
\end{equation}
For this solution, $u$ and $\rho$ are H\"older continuous on compact sets of $\T \times (0,\infty)$ (with H\"older exponent depending on the compact set).  Moreover, in the case where $f$ is compactly supported in time, the velocity field $u$ exhibits fast alignment to a constant.   
\end{THEOREM}

\begin{THEOREM}
\label{t:en}
Let $(u, \rho, e)$ be any weak solution on $[0,T]$, with $f\in L^2(\T\times [0,T])$.  If $\a\in (0,1)$, we assume additionally that $u\in L^3(0,T;B^{1/3}_{3,c_0})$ and $\rho\in L^3(0,T;B^{1/3}_{3,\infty})$.  If $\a\in [1,2)$, no additional assumption is needed.  Then $(u, \rho, e)$ satisfies the following energy equalities for a.e. $t\in [0,T]$:
\begin{equation}
\label{e:ee}
\frac12 \int_\T \rho u^2(t)\dx + \frac12 \int_0^t \int_\T \int_\R \rho(x)\rho(y) \frac{|u(x)-u(y)|^2}{|x-y|^{1+\a}}\dy\dx\ds = \frac12 \int_\T \rho_0 u_0^2\dx + \int_0^t \int_\T \rho u f \dx\ds
\end{equation}
\begin{equation}
\label{e:rhoee}
\int_\T \rho(t)^2\dx + \frac12\int_0^t \int_\T \int_\R (\rho(x)+\rho(y)) \frac{|\rho(x)-\rho(y)|^2}{|x-y|^{1+\a}}\dy\dx\ds = \int_\T \rho_0^2\dx - \int_0^t \int_\T e\rho^2\dx \ds.
\end{equation}
\end{THEOREM}

\begin{THEOREM}
\label{t:stng}
Let $(u_0, \rho_0, e_0)\in W^{1,\infty}\times W^{1,\infty}\times W^{1,\infty}$ satisfy the compatibility condition \eqref{e:compat}, and assume additionally that $\rho_0^{-1}\in L^\infty$, that $f\in L^\infty(0,\infty;W^{2,\infty})$, and that $\a\ne 1$. Then there exists a unique global-in-time strong solution $(u,\rho, e)$ associated to the initial data $(u_0, \rho_0, e_0)$. For this solution, $u$ and $\rho$ belong to $C_\loc((0,\infty); C^1)$.  Moreover, in the case where $f$ is compactly supported in time, $(u,\rho)$ undergoes fast flocking in $W^{1,\infty}\times L^\infty$ to some $(\overline{u}, \overline{\rho})\in \cF$.  In fact, the convergence occurs in $C^{1,\e}\times C^{1,\e}$ for some $\e>0$ (though we make no statement on the rate of convergence in this space).  Finally, even in the case $\a=1$, any strong solution is unique if it exists.  
\end{THEOREM}

\begin{REMARK}
It may seem somewhat strange that the case $\a=1$ should be excluded from our existence result on strong solutions.  The reason why our method does not yield existence of strong solutions in this case will be clear later from the estimates in Section \ref{ss:W1infty}; for now we simply note that the exclusion of the case $\a=1$ has some precedent.  In fact, the arguments of \cite{STI}, \cite{STIII} prove existence of solutions in $H^3\times H^{2+\a}$ for all $\a\in (0,2)\backslash\{1\}$; going up one more derivative is necessary only for $\a=1$.  It seems likely that our method could be applied to the case $\a=1$ (or other $\a$, for that matter) to yield solutions in $W^{2,\infty}$; however, we prefer to leave this case for future research.
\end{REMARK}

The outline of the paper is as follows.  In Section \ref{s:Reg}, we prove a priori bounds at the $L^\infty$ level for regular solutions.  Once these are established, the rest of the proof of Theorem \ref{t:Reg} follows the same steps as in \cite{STI}, \cite{STIII}, with trivial modifications.  We keep careful track of the dependencies of constants involved in these a priori bounds, in order to prove that they survive the limiting procedure we use to construct weak solutions in Section \ref{s:weak}.  Some additional bounds beyond those required for Theorem \ref{t:Reg} are needed to pass to the limit; we also include these in Section \ref{s:Reg}.  In Section \ref{s:en}, we prove Theorem \ref{t:en}.  In Section \ref{s:Wbds}, we continue proving bounds on regular solutions at the $W^{1,\infty}$ level, and in Section \ref{s:strong} we use these bounds to prove the existence of strong solutions when $\a\ne 1$.  Section \ref{s:strong} also contains the proof of the rest of Theorem \ref{t:stng}.

\section{Existence of and Bounds on Regular Solutions}
\label{s:Reg}

The proof of local-in-time existence of regular solutions of \eqref{e:mainv}, \eqref{e:maind} follows the arguments of \cite{STI}, \cite{STIII} with trivial adjustments to account for the forcing term. The proof of global-in-time existence mostly carries over to the forced case. The only step that requires adjustment from the unforced argument is proving $L^\infty$ bounds on the quantities $u$, $\rho$, $\rho^{-1}$, and $e$ that do not blow up in finite time.  Once such bounds are established, the proof of global-in-time existence again requires only minor adjustments to the proof in the unforced case. In what follows we assume that $(u, \rho)$ is a regular solution and that $e = u' - \L_\a \rho$ and $q=e/\rho$.  For the purposes of proving global existence of regular solutions, we only need to show that $(u,\rho)$ remains bounded in $H^4 \times H^{3+\a}$ on bounded time intervals (or, effectively, we need to show that $u$, $\rho$, and $e$ remain bounded in $L^\infty$ on bounded time intervals); however, we will later construct weak solutions as limits of regular solutions, and in order to ensure that our $L^\infty$ bounds survive the limiting process, we track the dependencies of these bounds carefully and relate each to the initial data.  The limiting process we use also requires some kind of compactness, which we satisfy by proving bounds in H\"older spaces.  We also derive the energy laws \eqref{e:ee} and \eqref{e:rhoee} that are satisfied by regular solutions.  These equalities in particular show that $u$ and $\rho$ are bounded in $L^2 H^{\a/2}$ on finite time intervals, with bounds depending only on the $L^\infty$ norms of the initial data and some other fixed quantities.  
	
\subsection{$L^\infty$ Bounds on First-Order Quantities}
\label{ss:FO}

We collect the bounds on $u$, $\rho$, and $e$ together in the following proposition.  

\begin{PROP}
\label{p:Linfty}
Let $(u, \rho)$ be a regular solution on the time interval $[0,\infty)$, and define $e$ and $q$ as above. The following bounds hold for some positive constants $c_i$, $i=0, 1,\ldots, 4$ and for all $x\in \T$, $t\ge 0$. 
\begin{equation}
\label{e:utinfty}
|u(x,t)|\le \|u_0\|_{L^\infty} + t\|f\|_{L^\infty_{x,t}}.
\end{equation}
\begin{equation}
\label{e:rhobds}
c_0 \exp(-c_1 t) \le \rho(x,t) \le c_3 \exp (c_4 t);
\end{equation}
\begin{equation}
\label{e:qupper2}
|q(x,t)| \le c_2 \exp(c_1 t).
\end{equation}    
\begin{equation}
\label{e:eupper}
|e(x,t)| \le c_2 c_3 \exp((c_1 + c_4)t).
\end{equation}
The constants $c_0$ and $c_2$ depend only on $\|\rho_0^{-1}\|_{L^\infty}$, $\|q_0\|_{L^\infty}$, $\cM$, and $\a$; the constant $c_1$ depends only on $\|f'\|_{L^\infty_{x,t}}$, $\cM$, and $\a$; the constant $c_3$ depends only on $\|\rho_0\|_{L^\infty}$, $c_2$, $\cM$, and $\a$; and $c_4$ depends only on $c_1$ and $\a$.  
\end{PROP}

\begin{REMARK}
\label{r:compare}
Our proof of the lower bound on the density is the same in spirit of the corresponding bound in \cite{KT} and uses a `breakthrough scenario' type argument.  We include the argument for our force $f$ in its entirety for the sake of completeness.  
\end{REMARK}

Throughout the proof of Proposition \ref{p:Linfty} (as well as in Section \ref{ss:W1infty} below), we will tacitly make use of the following application of the classical Rademacher Theorem:
\begin{LEMMA}
Suppose $g:\T\times \R_+\to \R$ is a Lipschitz function, such that for every $x\in \T$, the function $g(x,\cdot)$ is differentiable on all of $\R_+$.  For each $t\in \R_+$, let $x_+(t)$ and $x_-(t)$ denote points in $\T$ where $g(\cdot, t)$ achieves its maximum and minimum, respectively, and define $g_+(t) = g(x_+(t),t)$ and $g_-(t) = g(x_-(t),t)$.  Then 
\begin{itemize}
\item $g_+(t)$ and $g_-(t)$ are Lipschitz, with the same Lipschitz constant as $g$, and 
\item $\p_t g_+(t) = \p_t g(x_+(t), t)$ and $\p_t g_-(t) = \p_t g(x_-(t),t)$ for a.e. $t\in \R_+$. 
\end{itemize}	
\end{LEMMA}
For a proof of this precise statement, see for example the Appendix of \cite{CGSV2015}.  

\begin{proof}[Proof of Proposition \ref{p:Linfty}]
	
\textit{Part 1: Bounds on $u$}

Let $x_+(t)$ and $x_-(t)$ denote a maximum and a minimum, respectively, for $u(\cdot, t)$. 

We write the velocity equation as
\begin{equation}
\label{e:velocity}
u_t + uu' = -\L_\a(\rho u) + u\L_\a \rho + f = \int_\R \frac{u(\cdot +z)-u}{|z|^{1+\a}}\rho(\cdot +z)\dz + f.
\end{equation}
It is clear from the form of the integral in \eqref{e:velocity} and the nonnegativity of $\rho$ that 
\[
[-\L_\a(\rho u) + u\L_\a(\rho)](x_+(t), t) \le 0 \le [-\L_\a(\rho u) + u\L_\a(\rho)](x_-(t), t).
\]
This immediately implies that 
\[
[\p_t u - f](x_+(t), t)\le 0\le [\p_t u-f](x_-(t), t),
\]
so that 
\[
\p_t \|u(t)\|_{L^\infty} \le \|f\|_{L^{\infty}_{t,x}},
\]
from which \eqref{e:utinfty} follows by integration.

\vspace{2 mm}
\textit{Part 2: Lower Bound on $\rho$; Bounds on $q$}	

To avoid estimating $\rho$ in terms of derivatives of $u$, we rewrite the density equation as 
\begin{equation}
\rho_t + u\rho' = -q\rho^2 - \rho \L_\a \rho.
\end{equation}
Evaluating at a minimum $x_-(t)$ of $\rho(\cdot, t)$, we obtain 
\begin{equation}
\label{e:rholwrpre}
\p_t \rho (x_-(t), t) \ge -\|q(t)\|_{L^\infty}\rho_-(t)^2 + \rho_-(t)\L_\a \rho(x_-(t), t).
\end{equation}
Here $\rho_-(t)$ denotes the minimum of $\rho$ at time $t$.  (Below we will use the analogous notation $\rho_+(t)$ for the maximum at time $t$.) We now require a bound on $q$, which is feasible because of the transport equation \eqref{e:qtrans} that it satisfies.  We have
\begin{equation}
\label{e:qupper}
\|q(t)\|_{L^\infty} \le \|q_0\|_{L^\infty} + \|f'\|_{L^\infty_{x,t}}\int_0^t \rho_-(s)^{-1}\ds,
\end{equation}  
Now we can substitute \eqref{e:qupper} into \eqref{e:rholwrpre} to eliminate $q(t)$.  We obtain
\begin{equation}
\label{e:rholwrpre2}
\p_t \rho (x_-(t), t) \ge -\left[\|q_0\|_{L^\infty} + \|f'\|_{L^\infty_{x,t}}\int_0^t \rho_-(s)^{-1}\ds\right]\rho_-(t)^2 + \rho_-(t)\L_\a \rho(x_-(t),t).
\end{equation}
In order to establish the desired lower bound on $\rho$, we use \eqref{e:rholwrpre2} and argue by contradiction.  But let us first define the constants involved.  Denote $\iota(r) := \inf_{|x|<r} \phi_\a(x)$, where $\phi_\a$ is as in \eqref{e:kernel}. Note that $\iota(r)<\infty$ for all $r>0$ and that $\iota(\pi) = \inf_{x\in \T}\phi_\a(x)>0$.
Define 
\begin{equation}
\label{e:constants}
c_0 = \frac12\min\left\{ \rho_-(0),\; \frac{\iota(\pi)\cM}{2\pi\iota(\pi) + \|q_0\|_{L^\infty}} \right\},
\quad 
c_1 = \frac{2 \|f'\|_{L^\infty_{x,t}}}{\iota(\pi)\cM},
\end{equation}
We claim that the desired lower bound on $\rho$ holds for this choice of $c_0$, $c_1$.  Indeed, suppose that the lower bound in \eqref{e:rhobds} fails; then we can define $t_0:=\inf\{ t\ge 0 :\rho_-(t) = c_0 \exp(-c_1 t) \}$.  Clearly $t_0>0$, since $\rho_-(0)\ge 2 c_0$ by definition of $c_0$.  Furthermore, $\rho_-(s)\ge c_0 \exp(-c_1 s)$ for $s\in [0,t_0]$, so that
\begin{align*}
\|q_0\|_{L^\infty} + \|f'\|_{L^\infty_{x,t}}\int_0^{t_0} \rho_-(s)^{-1}\ds
& \le \|q_0\|_{L^\infty} + \|f'\|_{L^\infty_{x,t}}\int_0^t c_0^{-1} \exp(c_1 s)\ds \\
& = \|q_0\|_{L^\infty} + \frac{\iota(\pi)\cM}{2}[ \rho_-(t_0)^{-1} - c_0^{-1}] \\
& \le \|q_0\|_{L^\infty} + \frac{\iota(\pi)\cM}{2} \left[\rho_-(t_0)^{-1} - \frac{2(2\pi\iota(\pi) + \|q_0\|_{L^\infty})}{\iota(\pi)\cM}\right]\\
& = \frac{\iota(\pi)\cM}{2\rho_-(t_0)} - 2\pi \iota(\pi).
\end{align*}
We also have
\begin{align*}
\L_\a \rho(x_-(t_0), t_0) & = \int_\T \phi_\a(z)(\rho(x_-(t_0) + z,t_0) - \rho_-(t_0))\dz \\
& \ge \int_\T \iota(\pi)(\rho(x_- + z,t_0) - \rho_-(t_0))\dz\\
& = \iota(\pi) \cM - 2\pi \iota(\pi) \rho_-(t_0).
\end{align*}
Substituting the two previous estimates into \eqref{e:rholwrpre2} then gives us
\begin{align*}
\p_t \rho (x_-(t_0), t_0)
& \ge -\left[ \frac{\iota(\pi)\cM}{2\rho_-(t_0)} - 2\pi \iota(\pi) \right]\rho_-(t_0)^2 + \rho_-(t_0) [ \iota(\pi) \cM - 2\pi \iota(\pi) \rho_-(t_0) ]  \\
& = \frac{\rho_-(t_0)\iota(\pi)\cM}{2} >0.
\end{align*}
It follows that $\rho_-(s)<c_0\exp(-c_1 s)$ for some time $s<t_0$, contradicting our choice of $t_0$.  This proves the lower bound on $\rho$.   

The bound \eqref{e:qupper2} on $q$ is obtained, with $c_2 = \iota(\pi)\cM/(2c_0)$, by substituting the lower bound on $\rho$ into \eqref{e:qupper}.  

\vspace{2 mm}
\textit{Part 3: Upper Bound on $\rho$; Bounds on $e$}

For this bound, we exploit the singularity of the kernel, using the fact that $\limsup_{r\to 0} r\iota(r)=\infty$.  Recall that for $z\in [-\pi, \pi]\backslash \{0\}$, $\phi_\a(z)$ is defined as in \eqref{e:kernel}, so that for $r\in (0,\pi]$, we have 
\[
\iota(r) = \frac{1}{r^{1+\a}} + \sum_{k\in \N} \frac{1}{(2\pi k + r)^{1+\a}} + \sum_{k\in \N} \frac{1}{(2\pi k - r)^{1+\a}}.
\]
Both the sums in the equation above are bounded by some constant: $\iota(r) \le r^{-1-\a} + C$. Let $r_0\in (0,\pi)$ be such that $r_0^{-1-\a} = C$.  (This is certainly possible, because taking $r=\pi$ in the second sum above shows that $C>\pi^{-1-\a}$.) Then 
\begin{equation}
r^{-1-\a} \le \iota(r) \le 2 r^{-1-\a}, 
\quad \quad \text{ for }  r\in (0,r_0).
\end{equation}

Now we define 
\begin{equation}
\label{e:constants2}
c_3 = \max\left\{ 2\|\rho_0 \|_{L^\infty_x}, 4\cM (2c_2)^{\frac{1}{\a}}, \frac{2}{c_2}\cM  r_0^{-1-\a} \right\},
\quad c_4 = \frac{1+\a}{\a}c_1.
\end{equation}

We claim that the desired upper bound for $\rho$ holds for this choice of $c_3$, $c_4$. Suppose that the upper bound of \eqref{e:rhobds} fails; then we can define $t_0:=\inf\{t\ge 0: \rho_+(t) = c_3\exp(c_4 t)\}$.  Clearly $t_0>0$, since $\rho_+(0) \le c_3/2$ by definition of $c_3$.  

Let $x_+(t_0)$ denote the $x$-value where the maximum of $\rho(\cdot, t_0)$ is achieved and put 
\[
r_1: = \min\left\{ \left( 2c_2 \exp(c_1 t_0)\right)^{-\frac{1}{\a}}, r_0 \right\}.
\]
Then 
\begin{align*}
\p_t \rho(x_+(t_0), t_0) 
& \le \|q(t_0)\|_{L^\infty} \rho_+(t_0)^2 + \rho_+(t_0) \int_{|z|<r_1} \phi_\a(z)(\rho(x_+(t_0) + z, t_0) - \rho_+(t_0))\dz \\	
& \le \|q(t_0)\|_{L^\infty} \rho_+(t_0)^2 + \iota(r_1) \rho_+(t_0)(\cM - r_1\rho_+(t_0)) \\
& = [\|q(t_0)\|_{L^\infty} - r_1\iota(r_1)]\rho_+(t_0)^2  + \iota(r_1)\cM \rho_+(t_0).
\end{align*}
By choice of $r_1$, we have 
\[
r_1 \iota(r_1) \ge r_1^{-\a} \ge 2c_2 \exp(c_1 t_0).
\]
Combining this with the upper bound in \eqref{e:qupper2}, we continue the estimate above:
\begin{align*}
\p_t \rho(x_+(t_0), t_0) 
& \le [c_2 \exp(c_1 t_0) - 2c_2 \exp(c_1 t_0)]\rho_+(t_0)^2  + \iota(r_1)\cM \rho_+(t_0) \\
& = [\iota(r_1)\cM - c_2 \exp(c_1 t_0)\rho_+(t_0)]\rho_+(t_0).
\end{align*}
But then by our choices of $r_0$, $r_1$, $c_3$, $c_4$, and $t_0$, we have
\begin{align*}
\iota(r_1)\cM 
& \le 2\left[ \min\left\{ \left( 2c_2 \exp(c_1 t_0)\right)^{-\frac{1}{\a}}, r_0 \right\} \right]^{-1-\a}\cdot \cM \\
& = 2\cM \max \left\{ ( 2c_2\cdot (2c_2)^{\frac{1}{\a}} \exp(c_4 t_0), r_0^{-1-\a} \right\}\\
& \le c_2 \max \left\{ 4 \cM \left( 2c_2\right)^{\frac{1}{\a}}, \frac{2}{c_2} \cM r_0^{-1-\a} \right\} \exp(c_4 t_0) \\
& \le c_2 c_3 \exp(c_4 t_0) = c_2 \rho_+(t_0),
\end{align*}
which implies that 
\[
\p_t \rho(x_+(t_0), t_0)   \le c_2 \rho_+(t_0)^2[1 - \exp(c_1 t_0)]<0,
\]
contradicting the definition of $t_0$.  This proves the upper bound on $\rho$.  

In light of the relation $e = q \rho$, the upper bound \eqref{e:eupper} on $e$ is obtained by multiplying the upper bounds for $q$ and $\rho$.  	
\end{proof}

As noted above, this essentially completes the proof of the existence of a global-in-time regular solution for any initial data $(u_0, \rho_0)\in H^4\times H^{3+\a}$ away from vacuum. 

\subsection{Energy Equality and Bounds in $L^2 H^{\a/2}$}
\label{ss:EE}
Next, we recall that regular solutions satisfy the certain energy equalities \eqref{e:ee}, \eqref{e:rhoee}.  Multiplying the velocity equation by $\rho$ and the density equation by $u$, then adding the results together, we obtain the momentum equation:
\begin{equation}
\label{e:momentum}
(\rho u)_t + (\rho u^2)' = -\rho\L_\a(\rho u) + \rho u \L_\a \rho + \rho f.
\end{equation}
Multiply \eqref{e:momentum} by $u$ and add $\rho u$ times the velocity equation.  The result is
\begin{equation}
(\rho u^2)_t + (\rho u^3)' = -2\rho u[\L_\a(\rho u) - u \L_\a \rho] + 2\rho u f,
\end{equation}
or, after integration, 
\begin{align*}
\frac{\mathrm{d}}{\dt} \int_\T \rho u^2 \dx 
& = -\int_\T 2\rho u \L_\a (\rho u) - 2\rho u^2 \L_\a \rho\dx + 2\int_\T \rho u f\dx \\
& = -\int_\T 2\rho u \L_\a (\rho u) - \rho u^2 \L_\a \rho - \rho\L_\a(\rho u^2)\dx + 2\int_\T \rho u f\dx \\
& = - \int_\T \int_\R \rho(x)\rho(y) \frac{|u(x)-u(y)|^2}{|x-y|^{1+\a}}\dy\dx + 2\int_\T \rho u f\dx.
\end{align*}
We used the self-adjointness of $\L_\a$ to pass from the first to the second line.  Integrating in time, we obtain the following energy equality:
\[
\frac12 \int_\T \rho u^2(t)\dx + \int_0^t \int_\T \int_\R \rho(x)\rho(y) \frac{|u(x)-u(y)|^2}{|x-y|^{1+\a}}\dy\dx\ds = \frac12 \int_\T \rho_0 u_0^2\dx + \int_0^t \int_\T \rho u f \dx\ds,
\]
which is \eqref{e:ee}.  This equality also proves that $u$ is bounded in $L^2(0,T; H^{\a/2})$ by a constant depending on  $\|u_0\|_{L^\infty}$, $\|\rho_0\|_{L^\infty}$, $\|\rho_0^{-1}\|_{L^\infty}$, $\|e_0\|_{L^\infty}$, $\|f\|_{L^\infty_t W^{1,\infty}_x}$, $\cM$, $T$, and $\a$. 

The derivation of \eqref{e:rhoee} is similar.  We start with the density equation, multiply by $\rho$, and integrate, obtaining
\begin{align*}
\frac12 \frac{\mathrm{d}}{\dt} \int_\T \rho(t)^2\dx 
& = - \int_\T (\rho u)' \rho \dx 
= \frac12 \int_\T (\rho^2)' u \dx 
= -\frac12 \int_\T \rho^2 (e + \L_\a \rho)\dx \\
& = -\frac12 \int_\T e\rho^2 + \rho^2 \L_\a \rho\dx. 
\end{align*}
Symmetrizing, we get
\[
\int_\T \rho^2 \L_\a \rho\dx
= \frac12 \int_\T \int_\R (\rho^2(x)-\rho^2(y)) \frac{\rho(x) - \rho(y)}{|x-y|^{1+\a}}\dy\dx 
= \frac12 \int_\T \int_\R (\rho(x)+\rho(y)) \frac{|\rho(x) - \rho(y)|^2}{|x-y|^{1+\a}}\dy\dx.
\]
Therefore 
\[
\int_\T \rho(t)^2\dx + \frac12\int_0^t \int_\T \int_\R (\rho(x)+\rho(y)) \frac{|\rho(x)-\rho(y)|^2}{|x-y|^{1+\a}}\dy\dx\ds = \int_\T \rho_0^2\dx - \int_0^t \int_\T e\rho^2\dx \ds, 
\]
which is \eqref{e:rhoee}.  Thus $\rho$ is also bounded in $L^2(0,T; H^{\a/2})$ by a constant depending only on $\|\rho_0\|_{L^\infty_x}$, $\|\rho_0^{-1}\|_{L^\infty_x}$, $\|e_0\|_{L^\infty_x}$, $\|f\|_{L^\infty_t W^{1,\infty}_x}$, $\cM$, $\a$, and $T$.  Finally, since $L^\infty\cap H^{\a/2}$ is an algebra, we have that $\rho u$ is bounded in $L^2(0,T;H^{\a/2})$, with a bound that depends only on $\|f\|_{L^\infty_t W^{1,\infty}_x}$, $\cM$, $\a$, $T$, and the $L^\infty$ norms of $u_0$, $\rho_0$, $\rho_0^{-1}$, and $e_0$.

\subsection{Bounds in H\"older Spaces}
\label{ss:Holder1}
Let $[\cdot]_{C^\g(\T)}$ denote the H\"older seminorm
\[
[g]_{C^\g(\T)} = \sup_{\substack{x,y\in \T \\ x\ne y}} \frac{|g(x)-g(y)|}{|x-y|^\g},
\]
for any $g\in C^\g(\T)$.  Below we will write $C^\g$ for $C^\g(\T)$, with the understanding that this will always denote the seminorm with respect to spatial variables only (not including time).  As mentioned above, our construction of weak solutions will require bounds on $u$ and $\rho$ in some H\"older space.  The precise statement that we use is recorded below:
\begin{PROP}
Let $(u, \rho)$ be a regular solution on the time interval $[0,\infty)$.  There exists $\g>0$ such that $\rho$, $m=\rho u$, and $u$ satisfy bounds of the form 
\begin{equation}
\label{e:rhoHolder2}
[\rho(t)]_{C^{\g}}\le t^{-{\g}/\a} C_T, \quad t\in (0,T];
\end{equation}
\begin{equation}
\label{e:mHolder2}
[m(t)]_{C^{\g}}\le t^{-{\g}/\a} C_T, \quad t\in (0,T];
\end{equation}
\begin{equation}
\label{e:uHolder}
[u(t)]_{C^{\g}}\le t^{-{\g}/\a} C_T, \quad t\in (0,T].
\end{equation}
The constants $C_T$ may depend on $\|f\|_{L^\infty_t W^{1,\infty}_x}$, $\cM$, $\a$, $T$, and the $L^\infty$ norms of $u_0$, $\rho_0$, $\rho_0^{-1}$, and $e_0$.  The number $\g$ ultimately depends only on these same quantities.
\end{PROP}
\begin{REMARK}
For the purposes of constructing a weak solution, the bounds 
\[
[u(t)]_{C^{\g}}\le C_{\d,T}, 
\quad [\rho(t)]_{C^{\g}}\le C_{\d,T},
\quad \quad 
t\in [\d,T].
\]  
would suffice.  (Here $C_{\d,T}$ is a constant that may depend on the same quantities as $C_T$ above, and also may depend on $\d$.) However, the explicit bound \eqref{e:rhoHolder2} will be used later to control $\rho'$.
\end{REMARK}

\begin{REMARK}
The main Theorem of \cite{Silvestre} plays a key role here.  In the case when $\a\in (0,1)$, the hypothesis of \cite{Silvestre} asks for the drift $u$ to be $C^{1-\a}_{x,t}$ (in both time and space).  However, all that is really used in the proof there is the H\"older regularity in space, uniformly in time, i.e. $L^\infty(0,T; C^{1-\a})$.  This is fortunate for us, since the latter is exactly the norm we can control for $u$, by the equality $e = u' - \L_\a\rho$.  Therefore, when we refer to the result of \cite{Silvestre} here and below, it should be understood that (when $\a\in (0,1)$) we consider the version of the statement with the relaxed hypothesis $u\in L^\infty(0,T;C^{1-\a})$.
\end{REMARK}

\begin{proof}
The bound \eqref{e:uHolder} follows from \eqref{e:rhoHolder2}, \eqref{e:mHolder2}, and the bounds from the previous subsection, since
\[
u(y) - u(x) = \rho(y)^{-1}[m(y)-m(x)] + u(x)\rho(y)^{-1}[\rho(x) - \rho(y)].
\]
As for \eqref{e:rhoHolder2} and \eqref{e:mHolder2}, we begin by writing the density and momentum equations in parabolic form:
\begin{equation}
\label{e:parabdens}
\rho_t + u\rho' +\rho \L_\a \rho = - e\rho.
\end{equation}
\begin{equation}
\label{e:parabmo}
m_t + u m' +\rho \L_\a m = - em + \rho f.
\end{equation}
The diffusion operator $\rho \L_\a$ for these equations has kernel $K(x,h,t) = \rho(x,t)|h|^{-1-\a}$, so they are of the type considered in \cite{Silvestre}.  The quantities $-e\rho$ and $-em + \rho f$ play the role of (bounded) forcing terms.  In order to apply the main result of \cite{Silvestre}, we split into two cases.  In the first case, we assume $\a\in (0,1)$.  In this case, we apply $\p_x^{-1}$ to the relation $u' = e + \L_\a \rho$, then take $C^{1-\a}$ norms, to conclude that $u(t)$ is bounded in $C^{1-\a}$, with bounds that depend only on $\|f\|_{L^\infty_t W^{1,\infty}_x}$, $\cM$, $\a$, $T$, and the $L^\infty$ norms of $u_0$, $\rho_0$, $\rho_0^{-1}$, and $e_0$.  (Notice that \eqref{e:uHolder} is actually trivially satisfied in this case.)  Therefore the hypothesis of the main theorem of \cite{Silvestre} is satisfied, and we may conclude that that there exists $\g>0$, depending only on $\|u\|_{L^\infty(0,T;C^{1-\a})}$, such that 
\begin{equation}
\label{e:rhoHolder}
[\rho(t)]_{C^{\g}} \le \frac{C_T}{t^{\g/\a}} (\|\rho\|_{L^\infty(\T\times (0,T))} + \|e\rho\|_{L^\infty(\T\times (0,T))} ),
\end{equation}
\begin{equation}
\label{e:mHolder}
[m(t)]_{C^{\g}} \le \frac{C_T}{t^{\g/\a}} (\|m\|_{L^\infty(\T\times (0,T))} + \|em - \rho f\|_{L^\infty(\T\times (0,T))} ),
\end{equation}	
where the constant $C_T$ here depends only on $T$ and  $\|u\|_{L^\infty(0,T; C^{1-\a})}$.  Absorbing the $L^\infty$ norms on the right hand sides of \eqref{e:rhoHolder} and \eqref{e:mHolder}, and recalling that $C^{1-\a}$ depends only on $\|f\|_{L^\infty_t W^{1,\infty}_x}$, $\cM$, $\a$, $T$, and the $L^\infty$ norms of $u_0$, $\rho_0$, $\rho_0^{-1}$, and $e_0$, we obtain \eqref{e:rhoHolder2} and \eqref{e:mHolder2}, with the claimed dependencies.

The situation is similar if $\a\in [1,2)$.  In this \cite{Silvestre} gives us $\g>0$, depending only on $\|u\|_{L^\infty(\T\times (0,T))}$, such that \eqref{e:rhoHolder} and \eqref{e:mHolder} continue to hold, with $C_T$ depending only on $T$ and $\|u\|_{L^\infty(\T\times (0,T))}$.  Since $\|u\|_{L^\infty(\T\times (0,T))}$ itself depends only on $\|f\|_{L^\infty_t W^{1,\infty}_x}$, $\cM$, $\a$, $T$, and the $L^\infty$ norms of $u_0$, $\rho_0$, $\rho_0^{-1}$, and $e_0$, this completes the proof of the second case.
\end{proof}

\begin{REMARK}
In the case $\a\in (1,2)$, we can apply $\p_x^{-\a}$ to the relation $\L_\a \rho = u' - e$, then take $C^{\a-1}$ norms, to conclude that $\rho$ is bounded in $C^{\a-1}$, with bounds that depend only on $\|f\|_{L^\infty_t W^{1,\infty}_x}$, $\cM$, $\a$, $T$, and the $L^\infty$ norms of $u_0$, $\rho_0$, $\rho_0^{-1}$, and $e_0$.  In particular, the bound \eqref{e:rhoHolder2} is trivially satisfied in this case.  
\end{REMARK}

\subsection{Bounds on Time Derivatives}
\label{ss:dt}

\begin{PROP}
Let $(u, \rho)$ be a regular solution on the time interval $[0,\infty)$ and let $e=u'-\L_\a \rho$.  Then for any $T>0$, $\p_t u$ is bounded in $L^2(0,T;H^{-\a/2})$. Furthermore, $\p_t \rho$ and $\p_t e$ are bounded in $L^\infty(0,T;H^{-1})$. In each case, the bounds depend only on $\|f\|_{L^\infty_t W^{1,\infty}_x}$, $\cM$, $\a$, $T$, and the $L^\infty$ norms of $u_0$, $\rho_0$, $\rho_0^{-1}$, and $e_0$.	
\end{PROP}

\begin{proof} 
For $\f\in H^{\a/2}(\T)$, we have
\[
\langle u(t), \f \rangle_{H^{-\a/2}\times H^{\a/2}} - \langle u(s), \f \rangle_{H^{-\a/2}\times H^{\a/2}} = \int_s^t \langle g(s), \f \rangle_{H^{-\a/2}\times H^{\a/2}}\ds,  
\]
where 
\[
\langle g(s), \f \rangle_{H^{-\a/2}\times H^{\a/2}} = \int_\T -ue(s)\f - \L_{\a/2}(\rho u)(s) \L_{\a/2}\f + f(s)\f\dx.
\]
Since 
\[
\|g(s)\|_{H^{-\a/2}} \le C[\|u(s)\|_{L^\infty} \|e(s)\|_{L^\infty} + \|\rho u(s)\|_{H^{\a/2}} + \|f(s)\|_{L^\infty}],
\]
the desired bound on $\p_t u$ thus follows from the results of Sections  \ref{ss:FO} and \ref{ss:EE}.  

For $\f\in H^{1}(\T)$, we have
\[
\langle \rho(t), \f \rangle_{H^{-1}\times H^1} - \langle \rho(s), \f \rangle_{H^{-1}\times H^1} = \int_s^t \int_\T \rho u(s) \f' \dx \ds.  
\]
Therefore 
\[
\|\p_t \rho (s)\|_{H^{-1}} \le C\|\rho u(s)\|_{L^\infty},
\]
so that the desired bound on $\p_t \rho$ follows from the results of Section  \ref{ss:FO}.  The bound for $\p_t e$ can be proved in the same way.  
\end{proof}

\section{Weak Solutions}
\label{s:weak}

\subsection{Properties of General Weak Solutions}
\label{ss:wkpropsgen}

Let $(u, \rho, e)$ be a weak solution on the time interval $[0,T]$ associated to the initial data $(u_0, \rho_0, e_0)\in L^\infty \times L^\infty \times L^\infty$.  The purpose of this section is to record three simple facts about such a general weak solution, namely
\begin{itemize}
	\item The quantity $e$ satisfies a weak form of \eqref{e:econs}.  That is, for all $\f \in C^\infty(\T\times [0,T])$ and a.e. $t\in [0,T]$, we have
	\begin{equation}
	\label{e:ewk}
	\int_\T e(t)\f(t)\dx - \int_\T e_0\f(0)\dx - \int_0^t \int_\T e \p_t \f \dx \ds = \int_0^t \int_\T ue \f' + f' \f \dx\ds.
	\end{equation}
	\item The solution $(u,\rho, e)$ converges weak-$*$ in $L^\infty$ to the initial data.
	\item The weak time derivative of $u$ is a well-defined element of $L^2(0,T;H^{-\a/2})$; the weak time derivatives of $\rho$ and $e$ are well-defined elements of $L^\infty(0,T; H^{-1})$.
\end{itemize}

To see that the first of these is true, note first that \eqref{e:compat2} implies that for all for all $\f \in C^\infty(\T\times [0,T])$ and a.e. $t\in [0,T]$, we have
\[
\int_\T e \f(t) + u \f'(t) + \rho \L_\a\f(t)  \dx = 0.
\]
For any $t\in [0,T]$ for which the above holds and any $\f\in C^\infty(\T\times [0,T])$, we have then that 
\begin{align*}
& \int_\T e(t)\f(t)\dx - \int_\T e_0\f(0)\dx \\
& = - \left[ \int_\T u \f'(t)\dx - \int_\T u_0 \f'(0)\dx\right] - \left[ \int_\T \rho \L_\a\f(t)\dx - \int_\T \rho_0 \L_\a\f(0)  \dx \right] \\
& = - \left[ \int_0^t \int_\T u\p_t \f' \dx \ds -ue\f' - \rho u \L_\a\f' + f\f'\dx\ds \right] - \left[ \int_0^t \int_\T \rho \p_t \L_\a \f \dx \ds + \rho u \L_\a \f'\dx\ds \right] \\
& = - \int_0^t \int_\T u(\p_t \f)' + \rho \L_\a(\p_t \f)\dx\ds + \int_0^t \int_\T ue \f' + f'\f\dx\ds \\
& = \int_0^t \int_\T e \p_t \f \dx\ds + \int_0^t \int_\T ue\f' + f'\f\dx\ds. 
\end{align*}
This proves \eqref{e:ewk}, for a.e. $t\in [0,T]$ and all $\f\in C^\infty(\T\times [0,T])$.

To observe the weak-$*$ convergence to the initial data, substitute any (time-independent) $\f\in C^\infty(\T)$ into the weak formulation \eqref{e:weakv}, \eqref{e:weakd} or into \eqref{e:ewk}.  Clearly $\int_\T (u(t)-u_0)\f\dx \to 0$ as $t\to 0^+$, since the right side of \eqref{e:weakv} is an integral from $0$ to $t$ of an integrable quantity.  Since $C^\infty(\T)$ is dense in $L^1(\T)$, we conclude that $\int_\T (u(t) - u_0)\f \dx \to 0$ as $t\to 0^+$, for any $\f\in L^1(\T)$, i.e. $u(t)\stackrel{*}{\rightharpoonup}u_0$ weak-$*$ in $L^\infty$, as $t\to 0^+$.  The situation is similar for $\rho$ and $e$.  

Finally, the statement regarding the time derivatives is proved in a manner similar to that of Section \ref{ss:dt}.  

\subsection{Construction of a Weak Solution}
\label{ss:wkexist}

In this section we construct a weak solution as a subsequential limit of regular solutions with mollified initial data, as the mollification parameter tends to zero.  The following version of the Aubin-Lions-Simon compactness Lemma (\cite{Simon}, c.f. Theorem II.5.16 in \cite{BoyerFabrie}) will allow us to use the bounds from Section \ref{s:Reg} to choose an appropriate subsequence.  
\begin{LEMMA}
	Let $X\subset Y\subset Z$ be Banach spaces, where the embedding $X\subset Y$ is compact and the embedding $Y\subset Z$ is continuous.  Assume $p,r\in [1,\infty]$, and define for $T>0$ the following space: 
	\[
	E = \{v\in L^p(0,T;X): \frac{\mathrm{d}v}{\mathrm{d}t}\in L^r(0,T;Z)\}.
	\]
	\begin{enumerate}
		\item If $p<\infty$, the embedding $E\subset L^p(0,T; Y)$ is compact.
		\item If $p=\infty$ and $r>1$, then the embedding $E\subset C([0,T],Y)$ is compact.
	\end{enumerate}
\end{LEMMA}	

Let $\g$ be as in Section \ref{ss:Holder1}.  In the notation of the Aubin-Lions-Simon Lemma, we set 
\[
X_T = C^{\g}, 
\quad 
Y = C^0,
\quad 
Z = H^{-1},
\quad 
E_{\d,T} = \{v\in L^\infty(\d,T;C^{\g}): \p_t v\in L^2(\d,T;H^{-1})\}.
\]
The conclusion of the Lemma is then that the embedding $E_{\d, T} \subset C([\d,T]; C^0)$ is compact for any $T>\d>0$.

Choose $(u_0, \rho_0, e_0)\in L^\infty \times L^\infty \times L^\infty$, satisfying $\rho_0^{-1}\in L^\infty$ and the compatibility condition \eqref{e:compat}.
Let $\eta\in C_c^\infty(\R)$ be a standard mollifier ($\int \eta = 1$, $\supp \eta\subset \{|x|\le 1\}$), and let $f_\e$ denote the convolution of $f$ by $\e^{-1}\eta(\e^{-1}\cdot)$: $f_\e(x)=\e^{-1}\int_\R \eta(\e^{-1}y)f(x-y)\dy$.  Let $(u^\e, \rho^\e)$ denote the global strong solution associated to the initial data $((u_0)_\e, (\rho_0)_\e)$ and let $e^\e = (u^\e)' - \L_\a \rho^\e$.  Note that $(e_0)_\e =  (u_0)'_\e - \L_\a (\rho_0)_\e = e^\e(0)$ automatically.

\begin{CLAIM}
	The sequences $u^\e$ and $\rho^{\e}$ are bounded in $E_{\d,T}$ for any $T>\d>0$.
\end{CLAIM}
\begin{proof}
	Fix $T>\d>0$. In order to prove the claim, one needs to prove the following two statements: 
	\begin{enumerate}
		\item $u^\e$ and $\rho^\e$ are bounded sequences of $L^\infty(\d,T;C^{\g})$.
		\item $\p_t u^\e$ and $\p_t \rho^\e$ are bounded sequences of $L^2(\d,T;H^{-1})$.
	\end{enumerate}
	We have essentially proved these statements already.  We provide the remaining details for half of the first statement only; the rest follows the same reasoning.
	
	Section \ref{ss:Holder1} establishes that the norm of $u^\e$ in $L^\infty(\d,T;C^{\g})$ can be bounded above by a quantity that depends only on $\|f\|_{L^\infty_t W^{1,\infty}_x}$, $\cM$, $\a$, $T$, $\d$, and the $L^\infty$ norms of $(u_0)_\e$, $(\rho_0)_\e$, $(\rho_0)_\e^{-1}$, and $(e_0)_\e$. But these $L^\infty$ norms are bounded by those of $u_0$, $\rho_0$, $\rho_0^{-1}$, and $e_0$, respectively, and the remaining quantities are fixed.  Therefore $u^\e$ is a bounded sequence in $L^\infty(\d,T;C^{\g})$.  
\end{proof}

Applying the Aubin-Lions-Simon Lemma, we can now choose a subsequence $\{\e_k\}$, tending to zero as $k\to \infty$, such that $u^{\e_k}$ and $\rho^{\e_k}$ converge (strongly) in $C([2^{-N},2^N];C^0)$, with $N$ any natural number.  Using a standard diagonal argument, we obtain a further subsequence, which we continue to denote by $\e_k$, such that $u^{\e_k}$ and $\rho^{\e_k}$ converge to functions $u$ and $\rho$, respectively, in $C_{\loc}((0,\infty);C^0)$. 

Using the same logic as in the Claim above, we also have that $e^\e$ is bounded in $L^\infty(\T\times [0,T])$ and both $u^\e$ and $\rho^\e$ are bounded in $L^2(0,T;H^{\a/2})$.  Therefore we may choose a further subsequence (still denoted $\e_k$) such that $e^{\e_k}$ converges weak-$*$ in $L^\infty(\T\times [0,T])$ to some $e\in L^\infty(\T\times [0,T])$, and so that $u^{\e_k}$ and $\rho^{\e_k}$ converge weakly in $L^2(0,T;H^{\a/2})$ to $u$ and $\rho$. Then we can use a diagonal argument as above to send $T\to \infty$. To summarize, there exists a subsequence $\{\e_k\}$ and a triple $(u, \rho, e)$, such that as $k\to \infty$, we have 
\[ 
u^{\e_k}\to u \text{ and } \rho^{\e_k} \to \rho \text{ strongly in } C_\loc((0,\infty); C^0);
\]
\[
u^{\e_k} \rightharpoonup u \text{ and } \rho^{\e_k} \rightharpoonup \rho \text{ weakly in } L^2_{\loc}(0,\infty; H^{\a/2});
\]
\[
e^{\e_k} \stackrel{*}{\rightharpoonup} e \text{ weak-}*\text{ in } L^\infty_{\loc}(\T\times [0,\infty)).
\]
Now $(u^{\e_k}, \rho^{\e_k})$ is a regular solution (therefore $(u^{\e_k}, \rho^{\e_k}, e^{\e_k})$ is a weak solution) for each $k$. We can therefore consider each term in each equation of the weak formulation and easily see that the above convergences guarantee that $(u,\rho, e)$ satisfies the weak formulation. This completes the existence part of Theorem \ref{t:wk}.  The construction gives H\"older continuity on compact sets of $\T\times (0,\infty)$.  Indeed, if $\g$ is the H\"older exponent associated to the interval $[0,T]$ as in Section \ref{ss:Holder1}, then for any $\widetilde{\g}\in (0,\g)$, the convergences $u^{\e_k}\to u$ and $\rho^{\e_k}\to \rho$ can be taken in $L^\infty(\d,T;C^{\widetilde{\g}})$ for any fixed $\d>0$.  

\begin{REMARK}
If $\a\ne 1$, slightly more information is available.  If $0<\a<1$, then the above construction can be modified slightly to give $u^{\e_k}\to u$ in  $C([0,\infty); C^{1-\a-\k})$, for any $\k\in (0,1-\a)$; if $1<\a<2$, then we can obtain $\rho^{\e_k}\to \rho$ in  $C([0,\infty); C^{\a-1-\k})$ for any $\k\in (0,\a-1)$. 
\end{REMARK}

\subsection{Energy Inequality for Constructed Solutions}
We now prove that the solutions constructed above satisfy \eqref{e:ei} and \eqref{e:rhoei}. To prove these inequalities, we essentially use the fact that they are true (with equality) for regular enough solutions, then pass to the limit $k\to \infty$ in the sequence $(u^{\e_k}, \rho^{\e_k}, e^{\e_k})$ from the proof of existence above.  However, since the solution behaves a little better away from time zero, we initially work on $[\d,t]$ for some $\d>0$.  We prove \eqref{e:ei} first.  We start with the equality 
\begin{equation}
\label{e:eekd}
\int_\T \rho^{\e_k} (u^{\e_k})^2(s)\dx\bigg|^t_\d \!+\! \int_\d^t \! \int_\T \int_\R \rho^{\e_k}(x)\rho^{\e_k}(y) \frac{|u^{\e_k}(x)-u^{\e_k}(y)|^2}{|x-y|^{1+\a}}\dy\dx\ds 
= 2\int_\d^t \int_\T \rho^{\e_k} u^{\e_k} f \dx\ds.
\end{equation}

The first term and the forcing term are easily seen to converge to their natural limits, by uniform convergence of $\rho^{\e_k}$, $u^{\e_k}$ on any time interval $[\d,T]$.  To deal with the second term on the left, we write 
\begin{equation}
\int_\d^t \int_\T \int_\R [\rho^{\e_k}(x)\rho^{\e_k}(y)-\rho(x)\rho(y)] \frac{|u^{\e_k}(x)-u^{\e_k}(y)|^2}{|x-y|^{1+\a}}\dy\dx\ds
\to 0,
\quad \text{ as }k\to \infty,
\end{equation}
which is valid by uniform convergence of $\rho^{\e_k}$ away from time zero, as well as the $L^2 H^{\a/2}$ bound on $u^{\e_k}$, which is uniform in $k$.  We also have
\begin{equation}
\int_\d^t \!\int_\T \!\int_\R \!\rho(x)\rho(y) \frac{|u(x) \!-\! u(y)|^2}{|x-y|^{1+\a}}\dy\dx\ds
\le \liminf_{k\to \infty} \int_\d^t \!\int_\T \!\int_\R \!\rho(x)\rho(y) \frac{|u^{\e_k}(x)\!-\!u^{\e_k}(y)|^2}{|x-y|^{1+\a}}\dy\dx\ds,
\end{equation}
by weak lower semicontinuity.  Taking limits in \eqref{e:eekd} thus yields
\begin{equation}
\label{e:eid}
\int_\T \rho u^2(t)\dx + \int_\d^t \int_\T \int_\R \rho(x)\rho(y) \frac{|u(x)-u(y)|^2}{|x-y|^{1+\a}}\dy\dx\ds \le \int_\T \rho u^2(\d)\dx + 2 \int_\d^t \int_\T \rho u f \dx\ds
\end{equation}
Next, we note that  
\begin{equation}
\label{e:ei0d}
\int_\T \rho^{\e_k} (u^{\e_k})^2(\d)\dx \le \int_\T (\rho_0)_{\e_k} (u_0)_{\e_k}^2\dx + 2 \int_0^\d \int_\T \rho^{\e_k} u^{\e_k} f \dx\ds.
\end{equation}
This is obtained from the energy equality for $(u^{\e_k}, \rho^{\e_k})$ on $[0,\d]$, by dropping the enstrophy term.  We can estimate the force term on the right by $C\d$, where $C$ is independent of $k$ and $\d$ (but may depend on $t$), and then take $k\to \infty$.  The term on the left converges to its natural limit for the same reason as above; the initial data term converges to its natural limit by standard properties of mollifiers.  We are left with 
\[
\int_\T \rho u^2(\d)\dx \le \int_\T \rho_0 u_0^2\dx + C\d.
\]
Combining this with \eqref{e:eid}, we obtain 
\begin{equation}
\label{e:eid2}
\int_\T \rho u^2(t)\dx +\! \int_\d^t \!\int_\T \int_\R \rho(x)\rho(y) \frac{|u(x)-u(y)|^2}{|x-y|^{1+\a}}\dy\dx\ds \le \int_\T \rho_0 u_0^2\dx + 2\int_\d^t \!\int_\T \rho u f \dx\ds +C\d.
\end{equation}
And now, taking $\d\to 0$ yields \eqref{e:ei}.

The inequality \eqref{e:rhoei} is proved in a very similar way.  The only difference in approach is for the last term, on $[\d,t]$. We write
\[
\int_\d^t \int_\T e^{\e_k}(\rho^{\e_k})^2 \dx\ds - \int_\d^t \int_\T e \rho^2 \dx\ds
= \int_\d^t \int_\T e^{\e_k} [(\rho^{\e_k})^2-\rho^2] \dx\ds + \int_\d^t \int_\T (e^{\e_k}-e)\rho^2 \dx\ds.
\]
We use uniform convergence of the $\rho^{\e_k}$ on $[\d,t]$ to treat the first term and weak-$*$ convergence of $e^{\e_k}$ to treat the second. This finishes the proof of the inequalities \eqref{e:ei}, \eqref{e:rhoei}.  

\subsection{The Case of a Compactly Supported Force: Considerations for the Constructed Solutions}

\label{ss:0forcewk}

If the force $f$ is identically zero, or, more generally, if it is compactly supported in time, then there are several implications for the solutions we have constructed.  We take a moment to collect a few of these.  

\begin{enumerate}
\item If $f\equiv 0$, then the constants $c_1$ and $c_4$ from \eqref{e:rhobds}--\eqref{e:eupper} are both zero, so that $u$, $\rho$, $\rho^{-1}$, and $e$ can all be bounded above for all time by constants.  If $f$ is compactly supported in time, then all these quantities are still uniformly bounded, but the constants we can use to bound them will be larger, due to the potential growth during the time interval where $f$ is supported.  These uniform bounds will survive the limiting process used to construct weak solutions.

\item As a consequence of the uniform boundedness of $u$, $\rho$, $\rho^{-1}$, and $e$, the quantity $\g$ from Section \ref{ss:Holder1} can be taken to be independent of $T$.  Thus, the H\"older regularization will survive the limiting process (with H\"older exponent $\g-\kappa$ for any $\k\in (0,\g)$).

\item As soon as the force is turned off, we have a fast alignment of the velocity field; that is, the velocity amplitude $A(t) = \max_{x,y}|u(x,t) - u(y,t)|$ decays exponentially fast for regular solutions.  In particular, the case of zero force gives
\[
A(t) \le A(0)e^{-\cM \iota(\pi) t},
\]
where $\iota(r)=\inf_{|x|<r}\phi_\a(x)$ and $\phi_\a$ is the kernel of $\L_\a$, as above.  See Lemma 1.1 of \cite{STII} for the short proof of this statement. Therefore the alignment survives the limiting process used to construct weak solutions, so that (the constructed) weak solutions also enjoy the alignment property if the force is compactly supported.  
\end{enumerate}

The constructed weak solutions do not possess quite enough regularity for us to prove that they experience flocking (which also requires convergence of the density profile) in the case of a compactly supported force; however, we will see that flocking occurs for strong solutions under the assumption of compactly supported force.

\section{Energy Balance for Weak Solutions}
\label{s:en}

In this section, we provide conditions which guarantee that the natural energy laws hold for weak solutions.  We emphasize that the criteria we consider apply to any weak solutions, not just those weak solutions which can be constructed as in the previous section.  

To begin with, we note that it turns out to be easier to work with a momentum-based equation when proving \eqref{e:ee}.  However, due to the limited regularity of our weak solutions, we must prove that such a formulation is valid for our solutions.  In this proof and those below, we will make use of Littlewood-Paley theory, for which we give some basic notation presently.  Additional notation will be introduced as needed.  

\subsection{Notation for Littlewood-Paley Projections and Besov Spaces}
For a given function $g$, we denote by $g_q$ the projection of $g$ onto the $q$th Littlewood-Paley component, $q\in \{-1,0\}\cup \N$.  See for example \cite{CCFS} for the (standard) definitions of these projections.  We use the notation 
\[
g_{\le Q} = \sum_{q=-1}^Q g_q;
\quad \quad 
g_{> Q} = \sum_{q=Q+1}^\infty g_q.
\]
The Besov norm $B^s_{p,r}(\T)$ is defined by 
\[
\|g\|_{B^s_{p,r}(\T)} = \left\| \l_q^s \|g_q\|_{L^p(\T)} \right\|_{\ell_q^r},
\]
where we denote $\l_q:=2^q$.  Here $p,r\in [1,\infty]$, $s\in \R$. The Besov space $B^s_{p,r}$ is the space of tempered distributions whose $B^s_{p,r}$ norm is finite.  And we denote by $B^s_{p,c_0}$ the subspace of $B^s_{p,\infty}$ consisting of those elements $g$ such that $\limsup_{q\to \infty} \l_q^s \|u_q\|_{L^p} = 0$.  Finally, we note that $H^s = B^s_{2,2}$.

\subsection{The Weak Momentum Equation}

For smooth functions $f$ and $g$, we define 
\[
\cT(f,g) = -\L_\a(fg) - g\L_\a f.
\]
When $(\rho, u,e)$ is a weak solution, we can make sense of the expression $\rho \cT(\rho, u)$ in a weak sense.  Define $X= H^{\a/2}\cap L^\infty$, and for each $s>0$, let $\rho \cT(\rho, u)(s)$ denote the element of $X^*$ given by 
\[
\langle \rho \cT(\rho, u), \f \rangle_{X^*, X} =  \int  -\L_{\a/2}(\rho u) \L_{\a/2}(\rho \f) + \L_{\a/2}(\rho) \L_{\a/2}(\rho u \f) \dx.
\]

\begin{PROP}
Let $(u, \rho, e)$ be a weak solution on the time interval $[0,T]$. Then for each $\f\in C^\infty(\T\times [0,T])$ and a.e. $t\in [0,T]$, we have that 
\begin{equation}
\label{e:moment}
\begin{split}
\int \rho u \f(t) & \dx - \int \rho_0 u_0 \f(0)\dx - \int_0^t \int \rho u \p_t \f(s)\dx\ds \\ & = \int_0^t \int \rho u^2 \f' \dx\ds + \int_0^t \langle \rho \cT(\rho, u), \f \rangle_{X^*, X}\ds + \int_0^t \int \rho f \f\dx\ds.
\end{split}
\end{equation}
\end{PROP}
\begin{proof}
Substitute the test function $(\rho_{\le Q} \f)_{\le Q}$ into the weak velocity equation. We obtain
\begin{equation}
\label{e:veleq}
\begin{split}
\int_\T & \rho_{\le Q} u_{\le Q}(t) \f(t)\dx - \int_\T (\rho_0)_{\le Q} (u_0)_{\le Q} \f(0)\dx - \int_0^t \int_\T u_{\le Q} (\p_t \rho_{\le Q} \f + \rho_{\le Q} \p_t \f) \dx \ds \\
& = \int_0^t \int_\T -(ue)_{\le Q} \rho_{\le Q}\f - (\rho u)_{\le Q} \L_\a(\rho_{\le Q}\f) + \rho_{\le Q}f_{\le Q} \f\dx\ds.
\end{split}
\end{equation}
Then substitute $(u_{\le Q} \f)_{\le Q}$ into the weak density equation:
\begin{equation}
\label{e:denseq}
\begin{split}
\int_\T & \rho_{\le Q} u_{\le Q}(t) \f(t)\dx - \int_\T (\rho_0)_{\le Q} (u_0)_{\le Q} \f(0)\dx - \int_0^t \int_\T \rho_{\le Q} (\p_t u_{\le Q} \f + u_{\le Q} \p_t \f) \dx \ds \\
& = \int_0^t \int_\T (\rho u)_{\le Q} (u'_{\le Q}\f + u_{\le Q} \f') \dx\ds.
\end{split}
\end{equation}
Finally, project the compatibility condition onto the first $Q$ modes:
\begin{equation}
\label{e:compateq}
e_{\le Q} = u'_{\le Q} - \L_\a \rho_{\le Q}
\end{equation}
We use \eqref{e:compateq} to eliminate $u'_{\le Q}$ from \eqref{e:denseq}, then we add the result to \eqref{e:veleq}. We obtain
\begin{align*}
& \int_\T \rho_{\le Q} u_{\le Q}(t) \f(t)\dx - \int_\T (\rho_0)_{\le Q} (u_0)_{\le Q} \f(0)\dx - \int_0^t \int_\T \rho_{\le Q} u_{\le Q} \p_t \f \dx \ds \\
& \hspace{5 mm} = \int_0^t \int_\T (\rho u)_{\le Q} u_{\le Q} \f'\dx \ds + \int_0^t \int [(\rho u)_{\le Q} e_{\le Q}-\rho_{\le Q} (ue)_{\le Q}]\f\dx\ds \\
& \quad \quad \quad + \int_0^t \int (\rho u)_{\le Q}[\f \L_\a \rho_{\le Q} - \L_\a(\rho_{\le Q} \f)]\dx\ds + \rho_{\le Q}f_{\le Q} \f\dx\ds.
\end{align*}
Note that we have used the product rule and the fundamental theorem of calculus to simplify the left side of this equation.  It should now be clear that each integral converges to its natural limit, so that the equation \eqref{e:moment} holds.
	
\end{proof}
\begin{REMARK}	
It seems likely that the converse direction is also true, i.e., that replacing \eqref{e:weakv} with \eqref{e:moment} should give an equivalent weak formulation.  To try to prove this, one might try the following strategy: Denote $U:=\rho_{\le Q}^{-1}(\rho u)_{\le Q}$ and substitute $(\rho_{\le Q}^{-1}\f)_{\le Q}$ into \eqref{e:moment}.  Subtract from this equation the result of substituting $\left(\rho_{\le Q}^{-1} U\f \right)_{\le Q}$ into \eqref{e:weakd}. After performing some manipulations, one obtains
\begin{equation*}
\begin{split}
& \int U \f(t) \dx - \int U \f(0)\dx - \int_0^t \int U \p_t \f(t)\dx \\ 
& = \int_0^t \int [(\rho u^2)_{\le Q}- (\rho u)_{\le Q} U] \left( \frac{\f}{\rho_{\le Q}} \right)' \dx\ds + \int_0^t \int (\rho \cT(\rho, u))_{\le Q} \frac{\f}{\rho_{\le Q}} - \f \cT(\rho_{\le Q}, u_{\le Q}) \dx\ds \\
& \quad + \int_0^t \int - \rho_{\le Q} u_{\le Q} \L_\a \f - u_{\le Q} e_{\le Q} \f + \frac{(\rho f)_{\le Q}}{\rho_{\le Q}} \f \dx\ds  + \frac12 \int_0^t (U^2 - u_{\le Q}^2) \f' \dx\ds.
\end{split}
\end{equation*}	
All integrals on the left side and the last two integrals on the right side obviously converge to the natural limits.  The second term on the right side also converges to zero, though this requires some work (involving computations similar to those of Section \ref{ss:dissconv}).  However, it appears that the first term on the right side requires some additional smoothness in order to pass to the limit; the Onsager-type assumption \eqref{e:besovee} below is sufficient.  Therefore, we can currently claim only that the two weak formulations are equivalent under this additional assumption. As noted below, \eqref{e:besovee} is automatically satisfied when $\a\ge 1$.
\end{REMARK}	

\subsection{The Energy Budget}
Let $E_{\le Q}(t)$ denote the energy associated to scales $\l_q$ for $q\le Q$, and let $E(t)$ denote the total energy:
\[
E_{\le Q}(t)=\frac12 \int \frac{(\rho u)_{\le Q}^2}{\rho_{\le Q}}(t)\dx;
\quad \quad 
E(t)=\frac12\int \rho u^2(t)\dx.
\]  
The energy budget relation at scales $q\le Q$ is as follows:
\begin{equation}
\label{e:ebudgtot}
E_{\le Q}(t) - E_{\le Q}(0) = \int_0^t \Pi_Q(s)\ds - \varepsilon_Q(t) + \int_0^t \int (\rho f)_{\le Q} \cdot U\dx\ds.
\end{equation}
Here $\Pi_Q(s)$ is the flux through scales of order $Q$ due to the nonlinearity, defined by 
\begin{equation} 
\label{e:flux}
\Pi_Q = \int F_Q(\rho, u) U' \dx,
\end{equation}
\begin{equation}
\label{e:commutator}
F_Q(\rho, u) = (\rho u^2)_{\le Q} - U (\rho u)_{\le Q}, 
\end{equation}
and $\varepsilon_Q$ and $\int_0^t \int (\rho f)_{\le Q}\cdot U\dx\ds$ represent the change in energy due to the local interactions and the external force, respectively, at scales $q\le Q$. Now $\varepsilon_Q$ is given by 
\begin{equation*}
\label{e:dissQ}
\varepsilon_Q(t) =
-\int_0^t \int_\T (\rho \cT(\rho, u))_{\le Q} U\dx\ds.  
\end{equation*}
Also denote 
\begin{equation*}
\label{eq:disstot}
\varepsilon(t) = \frac12 \int_0^t \int_\R  \int_\T \rho(x) \rho(y) \frac{|u(x) - u(y)|^2}{|x-y|^{1+\a}} \dx\dy\ds
\end{equation*}
We do not give a derivation of the energy budget relation here; however, \eqref{e:ebudgtot} can be derived following essentially the same procedure as in \cite{LS2016}.

We aim to show that for appropriate $(\rho, u)$ and all $t\in [0,T]$, we have (as $Q\to \infty$) that $E_{\le Q}(t)\to E(t)$, $\int_0^t \Pi_Q(s)\,ds\to 0$, $\varepsilon_Q(t)\to \varepsilon(t)$, and $\int_0^t\int (\rho f)_{\le Q} \cdot U\dx\ds \to \int_0^t \int \rho u\cdot f\dx\ds$.  These convergences will immediately imply that the  energy balance relation holds for $(\rho, u)$.

Now, it was already shown in \cite{LS2016} that $\int_0^t \Pi_Q(s)\dt\to 0$ as $Q\to \infty$ whenever 
\[
\rho\in L^4(0,T;B^{\frac13}_{4,\infty}), 
\quad 
u\in L^4(0,T;B^{\frac13}_{4,c_0}).
\]
We do not expect to improve on the smoothness parameter here, but we have a bit of additional information here, namely the fact that $u\in L^\infty L^\infty$.  We can consequently weaken the integrability assumptions; see below.  We also claim that  $\varepsilon_Q\to \varepsilon$ holds in fact for all weak solutions, since such solutions satisfy $\rho, u \in L^\infty L^\infty\cap L^2 H^{\a/2}$, which is really all that is needed in order to pass to the limit for this term.  Finally, \cite{LS2016} also shows that the term $\int_0^t \int (\rho f)_{\le Q} \cdot U\dx\ds$ converges to its natural limit, and we therefore omit a treatment of this term.  In the following two subsections, we will prove that the natural energy law \eqref{e:ee} holds under the assumption that
\begin{equation}
\label{e:besovee}
	\rho\in L^3(0,T;B^{\frac13}_{3,\infty}), 
	\quad 
	u\in L^3(0,T;B^{\frac13}_{3,c_0}).
\end{equation}
Now \eqref{e:besovee} is automatically satisfied if $\a\in [1,2)$, since $L^\infty L^\infty \cap L^2 H^{1/2}\subset L^3 B^{1/3}_{3,3}$ by interpolation.  Therefore we will prove (the much more difficult half of) Theorem \ref{t:en} over the course of the next two subsections.  The proof of the other half (actually, a more precise statement) is contained in Section \ref{ss:ebrho}.

\subsection{Conditional Convergence of the Nonlinear Term}

We recall some notation and a few facts from \cite{LS2016}.  Let $a\in [1,\infty]$, $s\in (0,1)$; let $f$ be a real-valued function. Define the following:
\[
K^s_q = \left\{ \begin{array}{lcl}
\l_q^{s-1}, 	 & & q\ge 0; \\
\l_q^{s},		 & & q<0;
\end{array}\right.
\hspace{5 mm}
d_{a,q}^s (f) = \l_q^s\|f_q \|_{L^a};
\hspace{5 mm}
D_{a,Q}^s(f) = \sum_{q=-1}^\infty K^s_{Q-q} d^s_{a,q}(f).
\]
Note in particular that
\begin{equation}
\label{eq:Ddsim}
\limsup_{Q\ra \infty} D^s_{a,Q}(f) \sim \limsup_{q\to \infty} d^s_{a,q}(f).
\end{equation}
where the similarity constant depends only on $s$. 

\begin{PROP}\label{p:biest}
	For $f\in B^s_{a,\infty}$, $g\in L^\infty$, $a\in [1,\infty]$, $s\in (0,1)$, we have the following estimates:
	\begin{align}
	\|f(\cdot - y) - f(\cdot)\|_a & \lesssim (\l_Q |y| + 1)\l_Q^{-s} D^s_{a,Q}(f) \\
	\label{e:endpt}
	\|(fg)_{\le Q} - f_{\le Q} g_{\le Q} \|_{a} & \lesssim \l_Q^{-s} D_{a,Q}^s(f) \|g\|_\infty\\
	\label{e:gradflow}
	\| f_{\le Q}' \|_a 
	& \lesssim \l_Q^{1 - s} D^s_{a,Q}(f),\\
	\label{e:highs}
	\|f_{>Q}\|_a & \le \l_Q^{-s} D_{a,Q}^s(f)
	\end{align}
	If additionally $h\in B^t_{b,\infty}$, $t\in (0,1)$, $b\in [1,\infty]$, $\frac1c = \frac1a + \frac1b$, then 
	\begin{equation}
	\label{e:comm}
	\|(fh)_{\le Q} - f_{\le Q} h_{\le Q} \|_{c}
	\lesssim \l_Q^{-s-t} D^s_{a,Q}(f) D^t_{b,Q}(h)
	\end{equation}
\end{PROP}

\begin{lemma}
$F_Q(\rho, u)$ can be written as
\begin{equation}
\label{e:FQdecomp}
\begin{split}
F_Q(\rho, u) & =  r_Q(\rho, u, u) - \frac{1}{\rho_{\le Q}}[(\rho u)_{\le Q} - \rho_{\le Q} u_{\le Q}]^2 + \rho_{> Q}u_{>Q}\otimes u_{>Q} \\ & \hspace{5 mm} + 2 [(\rho u)_{\le Q} - \rho_{\le Q} u_{\le Q}] u_{>Q} + \rho[(u^2)_{\le Q} - u_{\le Q}^2],
\end{split}
\end{equation}
where 
\[
r_Q(\rho, u, u) = \int \widetilde{h}_Q(y)[\rho(x-y) - \rho(x)][u(x-y) - u(x)]^2\dy,
\]
and $\widetilde{h}_Q$ is a Schwartz function.  
\end{lemma}

With these facts in hand, we are now in a position to prove that the nonlinear term vanishes under our hypotheses.  

\begin{PROP}
The quantity $F_Q(\rho, u)$ satisfies the bound 
\begin{equation}
\|F_Q(\rho, u)\|_{L^{3/2}} \lesssim \l_Q^{-2/3}(D^{1/3}_{3,Q}(u))^2
\end{equation}
whenever $u\in B^{1/3}_{3,\infty}$ and $\rho\in L^\infty$.
\end{PROP}
This bound is a consequence of the decomposition \eqref{e:FQdecomp} and the bounds of Proposition \ref{p:biest}.

\begin{THEOREM}
\label{t:noflux}
Suppose $u\in L^3 B^{1/3}_{3,c_0}$ and $\rho\in L^3 B^{1/3}_{3,\infty}$.  Then $\int_0^t \Pi_Q(s)\ds\to 0$ as $Q\to \infty$.
\end{THEOREM}
\begin{proof}
First we write
\[
U' = \frac{1}{\rho_{\le Q}}[(\rho u)'_{\le Q} - U \rho'_{\le Q}].
\]	
Since $L^\infty\cap B^{1/3}_{3,\infty}$ is an algebra, we have $\rho u\in L^3 B^{1/3}_{3,\infty}$.  Therefore
\[
\|U'\|_{L^3} \lesssim \l_Q^{2/3}[D^{1/3}_{3,Q}(\rho u) + D^{1/3}_{3,Q}(\rho)],
\]
by \eqref{e:gradflow}. So 
\[
\int_0^t F_Q(\rho, u)U'\ds \lesssim \int_0^t (D^{1/3}_{3,Q}(u))^2 [D^{1/3}_{3,Q}(\rho u) + D^{1/3}_{3,Q}(\rho)]\ds.
\]
By \eqref{eq:Ddsim}, the definition of $B^{1/3}_{3,c_0}$, and the dominated convergence theorem, we conclude that the integral tends to zero, as needed.
\end{proof}

\subsection{Unconditional Convergence of the Dissipation Term}
\label{ss:dissconv}
In this subsection, we prove the following:
\begin{THEOREM}
\label{t:endiss}
Any weak solution $(\rho, u)$ satisfies $\varepsilon_Q\to \varepsilon$, as $Q\to \infty$.  	
\end{THEOREM}

Since the dissipation term involves fractional derivatives, we introduce a modified version of the localization kernel that we recalled in the previous section.  Define
\[
\widetilde{K}_q = \left\{ \begin{array}{lcl}
\l_q^{-\a/2}, 	 & & q\ge 0; \\
\l_q^{\a/2},		 & & q<0;
\end{array}\right.
\hspace{5 mm}
\widetilde{d}_q (f) = \l_q^{\a/2} \|f_q \|_{L^2};
\hspace{5 mm}
\widetilde{D}_Q(f) = \sum_{q=-1}^\infty \widetilde{K}_{Q-q} \widetilde{d}_{q}(f).
\]
Note that
\begin{equation}
\label{eq:Ddsima}
\limsup_{Q\ra \infty} \widetilde{D}_{Q}(f) \sim \limsup_{q\to \infty} \widetilde{d}_{q}(f),
\end{equation}
where the similarity constant depends only on $\a$. 

\begin{PROP}\label{p:biesta}
For $f\in B^{\a/2}_{2,\infty}$, $0<\a<2$, we have the following estimates:
\begin{align}
\label{e:Laflow}
\| \L_\a f_{\le Q} \|_2 
& \lesssim \l_Q^{\a/2} \widetilde{D}_Q(f),\\
\label{e:highalph}
\|f_{>Q}\|_2 & \le \l_Q^{-\a/2} \widetilde{D}_Q(f).
\end{align}
\end{PROP}
The proofs of \eqref{e:Laflow} and \eqref{e:highalph} are extremely similar to those of \eqref{e:gradflow} and \eqref{e:highs}, respectively, and are omitted.

\begin{REMARK}
We will also repeatedly use the following basic facts without comment below:
\begin{enumerate}
\item If $\supp \widehat{f}\subset B_{\l_Q}(0)$, then  $\|\L_\a f\|_{L^2}\lesssim \l_Q^\a\|f\|_{L^2}$.  
\item For $f\in B^{\a/2}_{2,\infty}$, the inequalities in \eqref{e:Laflow} and \eqref{e:highalph} continue to hold when $f_{\le Q}$ and $f_{>Q}$ are replaced with $f_Q$.  That is, for such $f$, we have 
\[
\| \L_\a f_{Q} \|_2 
 \lesssim \l_Q^{\a/2} \widetilde{D}_Q(f),
\quad \quad 
\|f_{Q}\|_2 \le \l_Q^{-\a/2} \widetilde{D}_Q(f).
\]
\item For $k\in \Z$, we have 
\[
\widetilde{D}_Q(f)\sim \widetilde{D}_{Q+k}(f),
\]
with the similarity constant depending only on $k$.  (To see this, simply note that $\widetilde{K}_{q+k}\sim \widetilde{K}_q$ for each $q\in \Z$, with a similarity constant that depends on $k$ but not on $q$.) 
\end{enumerate}
Of course, analogous properties hold when we consider first derivatives instead of fractional derivatives, but the fractional case is the one which is relevant below.
\end{REMARK}

To prove Theorem \ref{t:endiss}, we write 
\[
|\varepsilon_Q(t) - \varepsilon(t)|
\le \left| \varepsilon_Q(t) + \int_0^t \int \rho_{\le Q} u_{\le Q} \cT(\rho_{\le Q}, u_{\le Q})\dx\ds \right| + \left| \int_0^t \int \rho_{\le Q} u_{\le Q} \cT(\rho_{\le Q}, u_{\le Q})\dx\ds + \varepsilon(t)\right|,
\]
and we show that both terms tend to zero as $Q\to \infty$.  Let us take care of the (much easier) second term presently.  We write
\begin{align*}
& \int_0^t \int_\R \int_\T \rho_{\le Q}(x) \rho_{\le Q}(y) \frac{|u_{\le Q}(x) - u_{\le Q}(y)|^2}{|x-y|^{1+\a}}\dx\dy\ds
- \int_0^t \int_\R \int_\T \rho(x) \rho(y) \frac{|u(x) - u(y)|^2}{|x-y|^{1+\a} }\dx\dy\ds \\
& = \int_0^t \int_\R \int_\T [\rho_{\le Q}(x) \rho_{\le Q}(y) - \rho(x)\rho(y)] \frac{|u(x) - u(y)|^2}{|x-y|^{1+\a}}\dx\dy \ds \\
& \hspace{5 mm} + \int_0^t \int_\R \int_\T \rho_{\le Q}(x) \rho_{\le Q}(y) \frac{[(u_{\le Q}-u)(x) - (u_{\le Q}-u)(y)][(u_{\le Q}+u)(x) - (u_{\le Q}+u)(y)]}{|x-y|^{1+\a} }\dx\dy \ds.
\end{align*}
The first term here tends to zero by the dominated convergence theorem (the dominating function being $C\|\rho\|_{L^\infty}^2 \frac{|u(x) - u(y)|^2}{|x-y|^{1+\a}}$), while the second term is bounded above by 
\[
\int_0^t \|\rho\|_{L^\infty}^2 \|u_{\le Q} - u\|_{H^{\a/2}} \|u_{\le Q} + u\|_{H^{\a/2}}\ds \to 0,
\]
which tends to zero as $Q\to \infty$.  It thus remains to show that 
\[
\left| \varepsilon_Q(t) + \int_0^t \int \rho_{\le Q} u_{\le Q} \cT(\rho_{\le Q}, u_{\le Q})\dx\ds\right| \to 0, \quad \text{ as } Q\to \infty. 
\]
We write the relevant difference as
\begin{align*}
\int \rho \cT(\rho, u)U_{\le Q} - \rho_{\le Q} u_{\le Q} \cT(\rho_{\le Q}, u_{\le Q})\dx
& = \int \rho_{\le Q}u_{\le Q}\L_\a(\rho_{\le Q} u_{\le Q}) -(\rho\L_\a(\rho u))_{\le Q}U \dx \\
& \quad + \int (\rho u \L_\a \rho)_{\le Q}U - \rho_{\le Q} u_{\le Q}^2 \L_\a \rho_{\le Q}\dx\\
& =: A + B.
\end{align*}
Expanding further gives
\begin{align*}
A & = \int [\rho_{\le Q} u_{\le Q}-(\rho u)_{\le Q}]\L_\a(\rho_{\le Q} u_{\le Q}+(\rho u)_{\le Q})\dx + \int [\rho_{\le Q} \L_\a(\rho u)_{\le Q} - (\rho \L_\a(\rho u))_{\le Q}]U \dx \\
& =:A_1 + A_2.
\end{align*}
\begin{align*}
B & = \int [(\rho u\L_\a \rho)_{\le Q} - (\rho u)_{\le Q} \L_\a \rho_{\le Q}] U\dx + \int \rho_{\le Q}^{-1}[(\rho u)_{\le Q} - \rho_{\le Q} u_{\le Q}][(\rho u)_{\le Q} + \rho_{\le Q} u_{\le Q}]\L_\a \rho_{\le Q}\dx \\
& =:B_1 + B_2.
\end{align*}
The terms $A_1$ and $B_2$ are easy to treat:
\begin{align*}
A_1 & \lesssim \| \rho_{\le Q} u_{\le Q} - (\rho u)_{\le Q}\|_{L^2} \cdot \|\L_\a(\rho_{\le Q}u_{\le Q} + (\rho u)_{\le Q})\|_{L^2} \\
& \le \l_Q^{-\a} D_{2,Q}^{\a/2}(\rho) D_{2,Q}^{\a/2}(u)\cdot \l_Q^{\a}\|\rho_{\le Q}u_{\le Q} + (\rho u)_{\le Q}\|_{L^2} \\ 
& \lesssim D_{2,Q}^{\a/2}(\rho) D_{2,Q}^{\a/2}(u);\\
B_2 & \le \|\rho_{\le Q}^{-1}\|_{L^\infty} \|(\rho u)_{\le Q} - \rho_{\le Q} u_{\le Q}\|_{L^2} \|(\rho u)_{\le Q} + \rho_{\le Q} u_{\le Q}\|_{L^\infty} \|\L_\a \rho_{\le Q}\|_{L^2} \\
& \le C\cdot \l_Q^{-\a}  D_{2,Q}^{\a/2}(\rho) D_{2,Q}^{\a/2}(u) \cdot C \cdot \l_Q^\a \|\rho_{\le Q}\|_{L^2} \\
& \lesssim D_{2,Q}^{\a/2}(\rho) D_{2,Q}^{\a/2}(u).
\end{align*}
Thus 
\[
A_1 + B_2 \lesssim D_{2,Q}^{\a/2}(\rho)D_{2,Q}^{\a/2}(u),
\]
where the implied constant is independent of $Q$ (but may depend on the $L^\infty$ norms of $\rho$ and $u$).  

To deal with $A_2$ and $B_1$, we need to work more: Since $\rho\L_\a(\rho u)$ and $\rho u\L_\a \rho$ are in general only $L^2 H^{-\a/2}$, the commutator estimate \eqref{e:comm} in Proposition \ref{p:biest} does not directly apply.  To overcome this difficulty, we decompose the commutator $(f\L_\a g)_{\le Q}-f_{\le Q} \L_\a g_{\le Q}$ in such a way that repeated use of \eqref{e:Laflow} and \eqref{e:highalph} (and related inequalities) becomes an adequate substitute for \eqref{e:comm} in the treatment of $A_2$ and $B_1$.  Actually, we state our decomposition for the more general commutator $(fg)_{\le Q} - f_{\le Q} g_{\le Q}$, with the idea that $g$ will be replaced by $\L_\a g$ below.

We set the notation 
\[
f_{q_+} = f_{q+1} + f_{q+2},
\quad \quad 
f_{r_-} = f_{r-2} + f_{r-1} 
\quad \quad 
(q\ge -1,\; r\ge 1).
\]
\begin{LEMMA}
	The following decomposition holds:
	\begin{align*}
	(fg)_{\le Q} - f_{\le Q} g_{\le Q}
	& = \sum_{q>Q+2} [ f_q g_{q_-} + f_{q_-} g_q  + f_q g_q]_{\le Q} + [f_{Q_+} g_{\le Q} + f_{\le Q+2} g_{Q_+}]_{\le Q} \\
	& \hspace{5 mm} - [f_{(Q-2)_+} g_{\le Q-2} + f_{\le Q} g_{(Q-2)_+} ]_{Q+1} - [f_{Q} g_{\le Q-1} + f_{\le Q} g_{Q}]_{Q+2}.
	\end{align*}	
\end{LEMMA}
\begin{proof}
	Notice that if $p$ or $r$ is greater than $Q+2$ and $|p-r|>2$, then the Fourier support of $f_p g_r$ lies outside the ball of radius $\l_{Q+1}$ centered at $0$.  In particular, $(f_p g_r)_{\le Q}$ vanishes. Therefore 
	\[
	(fg)_{\le Q} = (f_{\le Q+2}g_{\le Q+2})_{\le Q} + \sum_{\substack{\max\{p,r\}>Q+2 \\ |p-r|\le 2}} (f_p g_r)_{\le Q}.
	\]
	So 
\begin{align*}
(fg)_{\le Q} - f_{\le Q} g_{\le Q}
& = [(fg)_{\le Q} - (f_{\le Q+2} g_{\le Q+2})_{\le Q}] + [(f_{\le Q+2}g_{\le Q+2})_{\le Q} - f_{\le Q} g_{\le Q}] \\
& = \sum_{\substack{\max\{p,r\}>Q+2 \\ |p-r|\le 2}} (f_p g_r)_{\le Q} + [f_{\le Q+2}g_{\le Q+2} - f_{\le Q} g_{\le Q}]_{\le Q} - (f_{\le Q} g_{\le Q})_{>Q}.
\end{align*}
We have the somewhat more explicit representation for the sum:
\begin{equation}
\label{e:ppsum}
\sum_{\substack{\max\{p,r\}>Q+2 \\ |p-r|\le 2}} (f_p g_r)_{\le Q}
= \sum_{q>Q+2} [ f_q g_{q_-} + f_{q_-} g_q  + f_q g_q]_{\le Q}.
\end{equation}
Writing $f_{\le Q+2} = f_{\le Q} + f_{Q_+}$ (and similarly for $g_{\le Q+2})$, then expanding $f_{\le Q+2}g_{\le Q+2}$, we obtain
\begin{equation}
\label{e:ppI}
(f_{\le Q+2}g_{\le Q+2})_{\le Q} - (f_{\le Q} g_{\le Q})_{\le Q} = [f_{Q_+} g_{\le Q} + f_{\le Q+2} g_{Q_+}]_{\le Q}.
\end{equation}
Finally, we note that $(f_p g_r)_q = 0$ whenever $\max\{p+2,r+2\}<q$.  Therefore
\begin{equation}
\label{e:ppllh}
\begin{split}
(f_{\le Q} g_{\le Q})_{>Q} =
[f_{(Q-2)_+} g_{\le Q-2} + f_{\le Q} g_{(Q-2)_+}]_{Q+1}
+ [ f_{Q} g_{\le Q-1} + f_{\le Q} g_{Q}]_{Q+2},
\end{split}
\end{equation}
Summing up the right hand sides of \eqref{e:ppsum} and \eqref{e:ppI}, then subtracting the right hand side of \eqref{e:ppllh}, we thus obtain the desired decomposition.
\end{proof}

\begin{PROP}
	Let $(f,g)$ be either $(\rho, \rho u)$ or $(\rho u, \rho)$.  Then 
	\begin{align*}
	\left| \int [(f\L_\a g)_{\le Q} - f_{\le Q} \L_\a g_{\le Q}]U\dx \right|
	& \lesssim  \left[ \sum_{q > Q} \l_q^\a \|f_q \|_{L^2}^2 \right]^{\frac12}  \left[ \sum_{q > Q} \l_q^\a \|g_q \|_{L^2}^2 \right]^{\frac12} + D^{\a/2}_{2,Q}(\rho) D^{\a/2}_{2,Q}(u)\\ 
	& \hspace{5 mm} + [ \widetilde{D}_Q(fu) + \widetilde{D}_Q(f) + \widetilde{D}_Q(u)] \widetilde{D}_Q(g).
	\end{align*}
\end{PROP}
\begin{proof}
	
We replace and $g$ with $\L_\a g$ in the decomposition of the Lemma, then multiply by $U$ and integrate. 
\begin{align*}
\int [(f \L_\a g)_{\le Q} - f_{\le Q} \L_\a g_{\le Q}] U\dx
& = \int U_{\le Q} \sum_{q>Q+2} [ f_q \L_\a g_{q_-} + f_{q_-} \L_\a g_q  + f_q \L_\a g_q]\dx\\
& \hspace{5 mm} + \int U_{\le Q} f_{\le Q+2} \L_\a g_{Q_+} \dx \\
& \hspace{5 mm} + \int  [U_{\le Q} f_{Q_+} \L_\a g_{\le Q}  - U_{Q+1} f_{(Q-2)_+} \L_\a g_{\le Q-2}  - U_{Q+2} f_{Q} \L_\a g_{\le Q-1}]\dx\\
& \hspace{5 mm} - \int [U_{Q+1} f_{\le Q} \L_\a g_{(Q-2)_+} + U_{Q+2} f_{\le Q} \L_\a g_Q ]\dx \\
& =: \mathrm{I} + \mathrm{II} + \mathrm{III} - \mathrm{IV}.
\end{align*}
Note that we have moved the outermost Littlewood-Paley projections onto the $U$'s and regrouped some terms.  
	
We estimate $\mathrm{I}$ and $\mathrm{II}$, as well as the first term in each of $\mathrm{III}$ and $\mathrm{IV}$.  The remaining terms in $\mathrm{III}$ and $\mathrm{IV}$ can be estimated similarly.
\begin{align*}
|\mathrm{I}| & \le \|U_{\le Q}\|_{L^\infty} \sum_{q>Q+2} [ \|f_q\|_{L^2} \|\L_\a g_{q_-}\|_{L^2} + \|f_{q_-}\|_{L^2} \|\L_\a g_q\|_{L^2}  + \|f_q\|_{L^2} \|\L_\a g_q \|_{L^2} ] \\
& \lesssim \sum_{q>Q+2} [ \|f_{q-2}\|_{L^2} +\|f_{q-1}\|_{L^2} + \|f_{q}\|_{L^2}]\cdot \l_q^\a [ \|g_{q-2}\|_{L^2} +\|g_{q-1}\|_{L^2} + \|g_q\|_{L^2}].
\end{align*}
Then by Cauchy-Schwarz, we conclude that 
\begin{equation}
\label{e:Ibd}
|\mathrm{I}| \lesssim \left[ \sum_{q > Q} \l_q^\a \|f_q \|_{L^2}^2 \right]^{\frac12}  \left[ \sum_{q > Q} \l_q^\a \|g_q \|_{L^2}^2 \right]^{\frac12}.
\end{equation}

The next term is the most troublesome.  We begin by rewriting $U$ as $\rho_{\le Q}^{-1}[(\rho u)_{\le Q} - \rho_{\le Q}u_{\le Q}] + u_{\le Q}$ and splitting the integral.
\begin{align*}
\mathrm{II} 
& = \int  \rho_{\le Q}^{-1} [(\rho u)_{\le Q} - \rho_{\le Q} u_{\le Q}]_{\le Q} f_{\le Q+2} \L_\a g_{Q_+}\dx + \int (u_{\le Q})_{\le Q} f_{\le Q+2} \L_\a g_{Q_+} \dx. 
\end{align*}
We bound the first term of $\mathrm{II}$ as follows:
\begin{align*}
\left| \int  \rho_{\le Q}^{-1}[(\rho u)_{\le Q}-u_{\le Q}]_{\le Q} f_{\le Q+2} \L_\a g_{Q_+}\dx\right| 
& \le \|\rho_{\le Q}\|_{L^\infty} \|(\rho u)_{\le Q} - \rho_{\le Q} u_{\le Q}\|_{L^2} \|f_{\le Q+2}\|_{L^\infty} \|\L_\a g_{Q_+}\|_{L^2} \\
& \le C\cdot C \l_Q^{-\a} D_{2,Q}^{\a/2}(\rho) D_{2,Q}^{\a/2}(u) \cdot C \cdot C\l_Q^{\a} \\
& \lesssim D_{2,Q}^{\a/2}(\rho) D_{2,Q}^{\a/2}(u).
\end{align*}
To estimate the second term, we recall that $g_{Q_+} = (g_{Q_+})_{>Q-1}$; we can then move the projection $>Q-1$ onto the other term $(u_{\le Q})_{\le Q} f_{\le Q+2}$ in the integrand: 
\[
\int (u_{\le Q})_{\le Q} f_{\le Q+2} \L_\a g_{Q_+} \dx = \int [(u_{\le Q})_{\le Q} f_{\le Q+2}]_{>Q-1} \L_\a g_{Q_+} \dx
\]
To see why this is useful, we need to massage the resulting expression a bit:
\begin{align*}
[(u_{\le Q})_{\le Q} f_{\le Q+2}]_{>Q-1}
& = [ (u_{\le Q} - (u_{\le Q})_{>Q})(f-f_{>Q+2}) ]_{>Q-1} \\
& = [ (u - u_{>Q} - (u_{>Q})_{\le Q})(f-f_{>Q+2}) ]_{>Q-1} \\
& = (fu)_{>Q-1} - (f_{>Q+2}u)_{>Q-1} - [(u_{>Q}+(u_{>Q})_{\le Q})f_{\le {Q+2}}]_{>Q-1}.
\end{align*}
The point is that, taking $L^2$ norms, we can now apply \eqref{e:highalph} to every term in this last expression above:
\begin{align*}
\|[(u_{\le Q})_{\le Q} f_{\le Q+2}]_{>Q-1}\|_{L^2}
& \le \|(fu)_{>Q-1}\|_{L^2}  + \|f_{>Q+2}\|_{L^2} \|u\|_{L^\infty} + 2\|u_{>Q}\|_{L^2} \|f\|_{L^\infty} \\
& \lesssim \l_Q^{-\a/2}[ \widetilde{D}_Q(fu) + \widetilde{D}_Q(f) + \widetilde{D}_Q(u) ]
\end{align*}
Thus  
\begin{align*}
\left| \int [(u_{\le Q})_{\le Q} f_{\le Q+2}] \L_\a g_{Q_+} \dx\right|
& \le \left\|[(u_{\le Q})_{\le Q} f_{\le Q+2}]_{>Q-1}\right\|_{L^2} \|\L_\a g_{Q_+} \|_{L^2} \\
& \lesssim [\widetilde{D}_Q(fu) + \widetilde{D}_Q(f) + \widetilde{D}_Q(u)] \widetilde{D}_Q(g).
\end{align*}
Overall, $\mathrm{II}$ is bounded by 
\begin{equation}
\label{e:IIbd}
|\mathrm{II}| \lesssim D_{2,Q}^{\a/2}(\rho) D_{2,Q}^{\a/2}(u) + [\widetilde{D}_Q(fu) + \widetilde{D}_Q(f) + \widetilde{D}_Q(u)] \widetilde{D}_Q(g).
\end{equation}

We estimate first term in $\mathrm{III}$ as follows:
\begin{align*}
\left| \int U_{\le Q} f_{Q_+} \L_\a g_{\le Q}\dx \right| 
& \le \|U\|_{L^\infty} \|f_{Q_+}\|_{L^2} \|\L_\a g_{\le Q}\|_{L^2} \\
& \le C\cdot C\l_Q^{-\a/2}\widetilde{D}_Q(f) \cdot \l_Q^{\a/2}\widetilde{D}_Q(g) \lesssim \widetilde{D}_Q(f) \widetilde{D}_Q(g)
\end{align*}
The second and third terms in $\mathrm{III}$ enjoy the same bound, which is proved the same way.  Thus
\begin{equation}
\label{e:IIIbd}
|\mathrm{III}|\lesssim \widetilde{D}_Q(f) \widetilde{D}_Q(g).
\end{equation}
Finally, the first term in $\mathrm{IV}$ is bounded by 
\begin{align*}
\left| \int U_{Q+1} f_{\le Q} \L_\a g_{(Q-2)_+}\dx \right|  
& = \left| \int [U-u_{\le Q}]_{Q+1} f_{\le Q} \L_\a g_{(Q-2)_+}\dx + \int (u_{Q+1})_{\le Q} f_{\le Q} \L_\a g_{(Q-2)_+}\dx \right| \\
& \lesssim D^{\a/2}_{2,Q}(\rho)D^{\a/2}_{2,Q}(u) + \widetilde{D}_Q(u) \widetilde{D}_Q(g).
\end{align*}
(The intermediate steps are all similar to those used for previous terms.)  And the other term in $\mathrm{IV}$ enjoys the same bound, so that 
\begin{equation}
\label{e:IVbd}
|\mathrm{IV}|\lesssim D^{\a/2}_{2,Q}(\rho)D^{\a/2}_{2,Q}(u) + \widetilde{D}_Q(u) \widetilde{D}_Q(g).
\end{equation}
Combining \eqref{e:Ibd}, \eqref{e:IIbd}, \eqref{e:IIIbd}, and \eqref{e:IVbd}, we obtain the desired statement.  	
\end{proof}

\begin{COROL}
We have 
\begin{align*}
& \hspace{- 15 mm} \left| \int \rho \cT(\rho, u)U_{\le Q} - \rho_{\le Q} u_{\le Q} \cT(\rho_{\le Q}, u_{\le Q})\dx \right| \\
& \lesssim \left[ \sum_{q > Q} \l_q^\a \|\rho_q \|_{L^2}^2 \right]^{\frac12}  \left[ \sum_{q > Q} \l_q^\a \|(\rho u)_q \|_{L^2}^2 \right]^{\frac12} + D^{\a/2}_{2,Q}(\rho) D^{\a/2}_{2,Q}(u) \\
& \hspace{5 mm}
+ [ \widetilde{D}_Q(\rho u) + \widetilde{D}_Q(\rho) + \widetilde{D}_Q(u)] \widetilde{D}_Q(\rho u) 
+ [ \widetilde{D}_Q(\rho u^2) + \widetilde{D}_Q(u)] \widetilde{D}_Q(\rho).
\end{align*}
Consequently, we have 
\begin{equation}
\label{e:LPdiss}
\int_0^t \left| \int \rho \cT(\rho, u)U_{\le Q} - \rho_{\le Q} u_{\le Q} \cT(\rho_{\le Q}, u_{\le Q})\dx \right|\dt \to 0,
\quad \text{ as } Q \to \infty.
\end{equation}
\end{COROL}
\begin{proof}
The displayed bound follows easily from the previous Proposition and the discussion at the beginning of this subsection. Indeed, recall that, in the notation from earlier, 
\[
\int \rho \cT(\rho, u)U_{\le Q} - \rho_{\le Q} u_{\le Q} \cT(\rho_{\le Q}, u_{\le Q})\dx = A_1 + A_2 + B_1 + B_2.
\]
We have already shown above that $A_1 + B_2 \lesssim D_{2,Q}^{\a/2}(\rho) D_{2,Q}^{\a/2}(u)$.  The Proposition gives us bounds for $A_2$ (with $(f,g) = (\rho, \rho u)$) and $B_1$ (with $(f,g) = (\rho u, \rho)$):
\[
A_2 \lesssim \left[ \sum_{q > Q} \l_q^\a \|\rho_q \|_{L^2}^2 \right]^{\frac12}  \left[ \sum_{q > Q} \l_q^\a \|(\rho u)_q \|_{L^2}^2 \right]^{\frac12} + D^{\a/2}_{2,Q}(\rho) D^{\a/2}_{2,Q}(u)
+ [ \widetilde{D}_Q(\rho u) + \widetilde{D}_Q(\rho) + \widetilde{D}_Q(u)] \widetilde{D}_Q(\rho u).
\]
\[
B_1 \lesssim \left[ \sum_{q > Q} \l_q^\a \|\rho_q \|_{L^2}^2 \right]^{\frac12}  \left[ \sum_{q > Q} \l_q^\a \|(\rho u)_q \|_{L^2}^2 \right]^{\frac12} + D^{\a/2}_{2,Q}(\rho) D^{\a/2}_{2,Q}(u)
+ [ \widetilde{D}_Q(\rho u^2) + \widetilde{D}_Q(\rho u) + \widetilde{D}_Q(u)] \widetilde{D}_Q(\rho). 
\]
Adding up the bounds on $A_1$, $A_2$, $B_1$, and $B_2$, we obtain the displayed estimate claimed in the Corollary.

The claimed limit then follows by the dominated convergence theorem, with dominating function 
$C[\|\rho\|_{H^{\a/2}}^2 + \|u\|_{H^{\a/2}}^2]$. 
\end{proof}
This completes the proof of Theorem \ref{t:endiss}.

\subsection{Energy Balance for the $\rho$ Equation}
\label{ss:ebrho}

It is not difficult to show that \eqref{e:rhoee} holds under the same assumptions as we proved for \eqref{e:ee}. Actually, something slightly more general is true:
\begin{PROP}
Let $(u, \rho, e)$ be a weak solution on $[0,T]$ and  assume that $u$ and $\rho$ satisfy
\begin{equation}
\label{e:besovrhoe}
\rho\in L^a(0,T;B^\s_{a,\infty}),
\quad 
u\in L^b(0,T;B^\t_{b,c_0}),
\quad \quad 
\frac2a + \frac1b = 2\s + \t = 1.
\end{equation}
Then \eqref{e:rhoee} holds for a.e. $t\in [0,T]$.
\end{PROP}

\begin{REMARK}
As suggested above, \eqref{e:besovee} is a special case of \eqref{e:besovrhoe}.  The reason the latter hypothesis is more flexible is that our proof of Theorem \ref{t:noflux} strongly depends on the fact that $\rho u\in B^{1/3}_{3,\infty}\cap L^\infty$ (by the algebra property of this space), whereas the argument of the present Proposition above requires information only about $\rho$ and $u$.  
\end{REMARK}

\begin{proof}
Substitute $(\rho_{\le Q})_{\le Q}$ into the weak density equation.  This gives	
\begin{align*}
\frac12\int_\T \rho(s)_{\le Q}^2 \dx\bigg|^t_0 
& = \int_0^t \int_\T (\rho u)_{\le Q} \rho_{\le Q}'\dx\ds \\
& = \int_0^t \int_\T [(\rho u)_{\le Q} - \rho_{\le Q} u_{\le Q} ] \rho_{\le Q}'\dx\ds  - \frac12 \int_0^t \int_\T \rho_{\le Q}^2 (e_{\le Q} + \L_a\rho_{\le Q}) \dx\ds.
\end{align*}
The first integral vanished as $Q\to \infty$, since
\[
\left| \int_0^t \int_\T [(\rho u)_{\le Q} - \rho_{\le Q} u_{\le Q} ] \rho_{\le Q}'\dx\ds \right|
\le \int_0^t D^s_{a,Q}(\rho)^2 D^t_{b,Q}(u) \ds.
\]
The other terms tend to their natural limits.  The only convergence which requires some justification is 
\[
\int_0^t \int_\T \rho_{\le Q}^2 \L_\a \rho_{\le Q}\dx\ds 
\to \frac12 \int_0^t \int_\T \int_\R (\rho(x)+\rho(y))\frac{|\rho(x) - \rho(y)|^2}{|x-y|^{1+\a}}\dy\dx\ds,
\text{ as } Q\to \infty.
\]
But this follows from an argument entirely similar to that of the convergence \eqref{e:LPdiss} which is proved above. We therefore omit the proof.
\end{proof}

\section{More Bounds on Regular Solutions: Toward a Theory of Strong Solutions}

\label{s:Wbds}

In order to construct strong solutions, we use essentially the same limiting process that we did for weak solutions.  In order to carry out this procedure up one level in regularity, we also give $L^\infty$ (and H\"older) bounds on $u'$, $\rho'$, and $e'$ below, and we track dependence of these bounds on the initial data as before.

\subsection{$L^\infty$ Bounds on Derivatives}
\label{ss:W1infty}

In this subsection, we prove $L^\infty$ bounds on $\rho'$, $u'$, and $e'$.  The density once again requires the most work.  We begin by eliminating all derivatives of $u$ from the $\rho'$ equation.  Recall that $u' = e + \L_\a \rho$.  Replacing $u''$ by $e' + \L_\a \rho'$ does not automatically eliminate the need to estimate derivatives of $u$, since the $e'$ equation involves $u'$.  Therefore we replace $e'$ with $q'$ via 
\begin{equation}
\label{e:exqx}
e' = \rho^2 \bigg( \frac{q'}{\rho} \bigg) + q \rho', 
\end{equation}
so that 
\begin{equation}
\rho u'' = \rho^3 \bigg( \frac{q'}{\rho} \bigg) + e \rho' + \rho\L_\a \rho'.
\end{equation}
This is a more satisfactory replacement in light of the transport equation \eqref{e:qxrhotrans} satisfied by $q'/\rho$.

In light of the above considerations, we write the $\rho'$ equation as 
\begin{equation}
\label{e:rhox}
\rho'_t + u\rho'' + \rho\L_\a \rho' = - \rho^3 \left( \frac{q'}{\rho} \right) - 2e\rho' - \rho' \L_\a \rho.
\end{equation}
We multiply by $\rho'$ and evaluate at a maximum $x_+(t)$ of $|\rho'(\cdot, t)|$, yielding
\begin{align*}
\frac12 \p_t [(\rho')^2](x_+(t),t) 
& = - \rho^3 \rho' \left( \frac{q'}{\rho} \right)(x_+(t),t) - 2e(\rho')^2(x_+(t),t) \\
& \quad \quad - (\rho')^2 \L_\a \rho(x_+(t),t) - \rho \rho' \L_\a \rho'(x_+(t),t)\\
& \le \|\rho(t)\|_{L^\infty}^3 \|\rho'(t)\|_{L^\infty} \left\| \frac{q'}{\rho}(t) \right\|_{L^\infty} + 2\|e(t)\|_{L^\infty} \|\rho'(t)\|_{L^\infty}^2 \\
& \quad \quad + |(\rho')^2 \L_\a \rho(x_+(t),t)| - \rho \rho' \L_\a \rho'(x_+(t),t).
\end{align*}
In order to bound $\|(q'/\rho)(t)\|_{L^\infty}$, we evaluate \eqref{e:qxrhotrans} at a maximum of $(q'/\rho)(\cdot, t)$, integrate in time, and substitute the previously obtained lower bound for $\rho$.  The result is
\begin{align*}
\left\| \frac{q'}{\rho}(t) \right\|_{L^\infty}
& \le \left\| \frac{q_0'}{\rho_0} \right\|_{L^\infty} + \int_0^t \left\| \frac{f''}{\rho^2} - \frac{f' \rho'}{\rho^3}\right\|_{L^\infty}\ds \\
& \le \|q_0'\|_{L^\infty} \|\rho_0^{-1}\|_{L^\infty} +  \frac{\|f''\|_{L^\infty_{x,t}}}{2 c_0^2 c_1} [\exp(2c_1 t)-1] + \frac{\|f'\|_{L^\infty_{x,t}}}{c_0^3} \int_0^t \exp(3c_1 s)\|\rho'(s)\|_{L^\infty}\ds.
\end{align*}
For the present purposes, the following rougher bound will suffice:
\begin{equation}
\label{e:qxuprough}
\left\| \frac{q'}{\rho}(t) \right\|_{L^\infty} \le C_T(\sup_{s\in [0,t]}\|\rho'(s)\|_{L^\infty} + 1), \quad t\in [0,T].
\end{equation}
Here $C_T$ is a constant that depends on $\|q_0\|_{W^{1,\infty}}$, $T$, $\|f\|_{L^\infty_t W^{2,\infty}_x}$, $\a$, $\cM$, and the $L^\infty$ norms of $\rho_0$, and $\rho_0^{-1}$.  However, for the remainder of this subsection, we will use $C_T$ to denote a constant that can depend on $\|f\|_{L^\infty_t W^{2,\infty}_x}$, $\cM$, $\a$, $T$, $\|\rho_0^{-1}\|_{L^\infty}$, and the $W^{1,\infty}$ norms of $u_0$, $\rho_0$, and $e_0$.  Recalling that $|\rho(x,t)|\le C_T$ and $|e(x,t)|\le C_T$, we have proved the following bound, which we pause to record as a Lemma.

\begin{LEMMA}
Let $(u, \rho)$ be a regular solution.  If $x_+(t)$ is a maximum of $|\rho'(\cdot, t)|$, then 
\begin{equation}
\label{e:drhobd1}
\frac12 \p_t [(\rho')^2](x_+(t),t) 
\le C_T(\sup_{s\in [0,t]}\|\rho'(s)\|_{L^\infty}^2 + 1) + |(\rho')^2 \L_\a \rho(x_+(t),t)| - \rho \rho' \L_\a \rho'(x_+(t),t),
\end{equation}
where $C_T$ is a constant that may depend on $\|f\|_{L^\infty_t W^{2,\infty}_x}$, $\cM$, $\a$, $T$, $\|\rho_0^{-1}\|_{L^\infty}$, and the $W^{1,\infty}$ norms of $u_0$, $\rho_0$, and $e_0$.
\end{LEMMA}

We now provide some bounds on the final two terms of the above inequality.  We will use the notation
\begin{equation}
\label{e:defDalpha}
D_\a g(y):= \int_\R \frac{|g(y)-g(y+z)|^2}{|z|^{1+\a}}\dz 
\end{equation}
for functions $g$ such that the integral makes sense.

Before proceeding, we make note of the following facts, which will be useful later: The following bounds hold for a maximum $x$ of $g'$ and some absolute constant $c_5$:
\begin{equation}
\label{e:dissterm}
 g' \L_\a g'(x) \ge D_\a g'(x)
\end{equation}
\begin{equation}
\label{e:nlmp}
D_\a g'(x)\ge c_5 \frac{|g'(x)|^{2+\a}}{\|g\|_{L^\infty}^\a}.
\end{equation}
Both of these follow from the nonlinear maximum principle of \cite{ConstVicol}.
	
The `bad' term in the inequality from the Proposition above is $|(\rho')^2 \L_\a\rho|$.  In order to estimate this term, we will use the following decomposition of the fractional Laplacian $\L_\a$:

\begin{LEMMA}
Let $\f\in C^\infty_c(\R)$ be even, identically $1$ on $[-1,1]$, and supported in $(-2,2)$. The following decomposition holds for sufficiently smooth $g$ and any $r>0$:
\begin{equation}
\label{e:Ldecomp}
\L_\a g(x) = \int_\R \frac{z}{\a}\f\left( \frac{z}{r} \right) \frac{g'(x) - g'(x+z)}{|z|^{1+\a}}\dz + \int_\R \left[ 1 - \f\left( \frac{z}{r} \right) + \frac{z}{\a r}\f'\left( \frac{z}{r} \right)\right] \frac{g(x) - g(x+z)}{|z|^{1+\a}}\dz.
\end{equation}	
Consequently, we have the following bounds:
\begin{equation}
\label{e:Lbdinfty}
|\L_\a g(x)| \le C r^{1 - \frac{\a}{2}} D_\a g'(x)^{\frac12} + C r^{-\a}\|g\|_{L^\infty}.
\end{equation}
\begin{equation}
\label{e:LbdHold}
|\L_\a g(x)| \le C r^{1 - \frac{\a}{2}} D_\a g'(x)^{\frac12} + C r^{\g-\a} [g]_{C^\g}.
\end{equation}
\end{LEMMA}	
\begin{proof}
Rewrite the right side of \eqref{e:Ldecomp} as follows:
\begin{equation}
\label{e:Ldecpf}
\begin{split}
\mathrm{RHS} & = 
\int_\R \frac{z \f\left( \frac{z}{r} \right) g'(x)}{\a|z|^{1+\a}} \dz + \int_\R \frac{z}{\a|z|^{1+\a}} \left[\frac{1}{r}\f'\left( \frac{z}{r} \right)(g(x) - g(x+z))  -\f\left( \frac{z}{r} \right) g'(x+z)\right] \dz \\
& \quad + \int_\R \left[ 1 - \f\left( \frac{z}{r} \right) \right] \frac{g(x) - g(x+z)}{|z|^{1+\a}}\dz.
\end{split}
\end{equation}
The first integral vanishes, while the second can be rewritten as 
\[
\int_\R \frac{z}{\a |z|^\a} \frac{\mathrm{d}}{\dz}\left[ \f\left( \frac{z}{r} \right) (g(x)-g(x+z)) \right]\dz = \int_\R \f\left( \frac{z}{r} \right) \frac{g(x)-g(x+z)}{|z|^{1+\a}}\dz,
\]
after integrating by parts.  Combining with the third term in \eqref{e:Ldecpf}, we obtain the usual integral formula for $\L_\a g$.  This completes the proof of \eqref{e:Ldecomp}. To obtain the inequality under consideration, use Cauchy-Schwarz on the first integral in \eqref{e:Ldecomp} and pull out the $L^\infty$ norm or $C^\g$ seminorm in the second.
\end{proof}
	
\begin{LEMMA}
Let $(u, \rho)$ be a regular solution on the time interval $[0,T]$.  The following bounds holds for a maximum $x_+(t)$ of $\rho'(\cdot, t)$ and some constant $C_T$ which may depend only on $\|f\|_{L^\infty_t W^{2,\infty}_x}$, $\cM$, $\a$, $T$, $\|\rho_0^{-1}\|_{L^\infty}$, and the $W^{1,\infty}$ norms of $u_0$, $\rho_0$, and $e_0$.
\begin{equation}
\label{e:badterm}
|(\rho')^2 \L_\a \rho(x_+(t),t)| \le C_T \|\rho'(t)\|_{L^\infty}^{2+\a} + \frac14 \rho_-(t) D_\a \rho'(x_+(t),t), \quad \quad t\in [0,T];
\end{equation}
\begin{equation}
\label{e:badtermH}
|(\rho')^2 \L_\a \rho(x_+(t),t)| \le C_T \frac{\|\rho'(t)\|_{L^\infty}^{2+\a-\g}}{ t^{\g/\a}} + \frac14 \rho_-(t) D_\a \rho'(x_+(t),t), \quad \quad t\in (0,T].
\end{equation}
\end{LEMMA}

\begin{proof}
We begin by putting $g=\rho$ in \eqref{e:LbdHold}. We use \eqref{e:rhoHolder2} for the first term; on the second we use Young's inequality, followed by \eqref{e:nlmp}.
\begin{align*}
|(\rho')^2 \L_\a \rho(x_+(t),t)| 
& \le Cr^{\g-\a} \|\rho'(t)\|_{L^\infty}^2 [\rho(t)]_{C^\g}  + D_\a \rho'(x_+(t), t)^{1/2} \cdot  C\|\rho'(t)\|_{L^\infty}^2 r^{1-\a/2} \\
& \le C_T r^{\g-\a}\frac{\|\rho'(t)\|_{L^\infty}^2}{t^{\g/\a}} + \frac{1}{8} \rho_-(t) D_\a \rho'(x_+(t), t) + c_6 \rho_-(t)^{-1} \|\rho'(t)\|_{L^\infty}^4 r^{2-\a} \\
& \le C_T \frac{\|\rho'\|_{L^\infty}^{2+\a-\g}}{t^{\g/\a}} + \frac{1}{8} \rho_-(t) D_\a \rho'(x_+(t), t) + \frac{c_5}{8} \frac{\rho_-(t)}{\rho_+(t)^\a} \|\rho'(t)\|_{L^\infty}^{2+\a} \\
& \le C_T \frac{\|\rho'\|_{L^\infty}^{2+\a-\g}}{t^{\g/\a}} + \frac{1}{4}\rho_-(t) D_\a \rho'(x_+(t), t),
\end{align*}
provided that we choose 
\[
r = \left[ \frac{c_5 \rho_-(t)^2}{8c_6 \rho_+(t)^\a} \right]^{\frac{1}{2-\a}} \|\rho'(t)\|_{L^\infty}^{-1},
\]
where $c_5$ is the constant from \eqref{e:nlmp}.

The inequality \eqref{e:badterm} is established similarly, starting with \eqref{e:Lbdinfty} instead of \eqref{e:LbdHold}. 
\end{proof}

\begin{THEOREM}
\label{t:drhobd}
Let $(u, \rho)$ be a regular solution on the time interval $[0,\infty)$.  For each $T>0$, there exists a constant $C_T^{\rho'}$, which may depend on $\|f\|_{L^\infty_t W^{2,\infty}_x}$, $\cM$, $\a$, $T$, $\|\rho_0^{-1}\|_{L^\infty}$, and the $W^{1,\infty}$ norms of $u_0$, $\rho_0$, and $e_0$,  such that 
\begin{equation}
\|\rho'(t)\|_{L^\infty} \le C_T^{\rho'},  \quad \quad t\in [0,T].
\end{equation}
\end{THEOREM}
\begin{proof}

Step 1: We find a time $t_*$ such that $|\rho'(t)|$ does not grow too much on the interval $[0,t_*]$; more specifically, 
\begin{equation}
\label{e:rhotsmall}
\|\rho'(t)\|_{L^\infty} \le 2 \|\rho_0'\|_{L^\infty}, \quad \text{ for } t\in [0,t_*].
\end{equation}

To this end, we apply \eqref{e:drhobd1}, \eqref{e:dissterm}, \eqref{e:badterm}, and \eqref{e:nlmp} to conclude that on the interval $t\in [0,1]$, we have 

\begin{align*}
\p_t \|\rho'(t)\|_{L^\infty}^2
& = \p_t [(\rho')^2](x_+(t), t) \\
& \le C_1[\|\rho'\|_{L^\infty(\T\times [0,t])}^2 + 1] + C_1 \|\rho'(t)\|_{L^\infty}^{2+\a} \\ 
& \le C_1 \|\rho'\|_{L^\infty(\T\times [0,t])}^{2+\a}.
\end{align*}
This implies that 
\[
\|\rho'(t)\|_{L^\infty}^2 \le \frac{\|\rho_0'\|_{L^\infty}^2}{\left[ 1 - \frac{C_1 \a \|\rho_0'\|_{L^\infty}^{\a}}{2} t \right]^{2/\a}}
\quad \text{ for } 0\le t< \frac{2}{C_1 \a |\rho_0'|_{L^\infty}^\a}.
\]
In particular, putting 
\[
t_* = \min\left\{ \frac{2(1-2^{-\a})}{C_1 \a \|\rho_0'\|_{L^\infty}^\a}, 1\right\},
\]
we obtain \eqref{e:rhotsmall}, as needed.  Note that $C_1$ depends only on $\|u_0\|_{L^\infty}$, $\|\rho_0\|_{L^\infty}$, $\|\rho_0^{-1}\|_{L^\infty}$, $\|q_0\|_{W^{1,\infty}_x}$, $\|f\|_{L^\infty_t W^{2,\infty}_x}$, $\a$, and $\cM$, so that $t_*$ depends only on these quantities and $\|\rho_0'\|_{L^\infty}$.   

Step 2: We obtain bounds on $\rho'$ for $t\ge t^*$, by using \eqref{e:badtermH}. The point is that in light of the reduction of the power $2+\a$ to $2+\a - \g$, we can now absorb the bad term into the dissipation. At a maximum $x_+(t)$ of $\rho'$, we have
\begin{align*}
\frac12 \p_t[(\rho')^2](x_+(t),t)
& \le C_T [\|\rho'\|_{L^\infty(\T\times [0,t])}^2 + 1] + |(\rho')^2 \L_\a \rho(x_+(t),t)| - \rho \rho' \L_\a \rho'(x_+(t),t)\\
& \le C_T [\|\rho'\|_{L^\infty(\T\times [0,t])}^2 + 1] + 2\bigg[ C_T \|\rho'(t)\|_{L^\infty}^{2+\a-\g}[\rho(t)]_{C^\g} + \frac{1}{4}\rho_-(t) D_\a \rho'(x_+(t), t) \bigg] \\
& \quad \quad - \rho_-(t) D_\a\rho'(x_+(t),t) \\
& \le C_T [\|\rho'\|_{L^\infty(\T\times [0,t])}^2 + t^{-\g/\a}\|\rho'(t)\|_{L^\infty}^{2+\a-\g} + 1] - c_T \|\rho'(t)\|_{L^\infty}^{2+\a}.
\end{align*}
We claim that this implies 
\[
\|\rho'(t)\|_{L^\infty} \le \max\left\{ \left( \frac{5 C_T}{c_T} \right)^{\frac{1}{\a}}, \left( \frac{2 C_T}{t_*^{\g/\a}c_T} \right)^{\frac{1}{\g}},  3\|\rho_0'\|_{L^\infty}, 1 \right\}=:C_T^{\rho'},
\quad \text{ for } t\in [0,T].
\]
Indeed, let $t_0$ be the largest possible time in the interval $[0,T]$ such that $\|\rho'(t)\|_{L^\infty} \le C_T^{\rho'}$ for all $t\in [0, t_0]$.  Then $t_0>t_*$ by Step 1 and the definition of $C_T^{\rho'}$.  Suppose that $t_*<t_0<T$.  Then $\|\rho'(t_0)\|_{L^\infty} = \sup_{t\in [0,t_0]}\|\rho'(t)\|_{L^\infty}= C_T^{\rho'}$, so that 
	
\begin{align*}
\frac12 \p_t[(\rho')^2](x_+(t),t)
& \le C_T \|\rho(t_0)\|_{L^\infty}^2 \left( 2 - \frac{c_T}{2 C_T}(C_T^{\rho'})^\a \right) + C_T \|\rho(t_0)\|_{L^\infty}^{2+\a-\g} \left( \frac{1}{t_0^{\g/\a}} - \frac{c_T}{2 C_T}(C_T^{\rho'})^\g \right)  \\
& \le C_T \|\rho(t_0)\|_{L^\infty}^2 \left( 2 - \frac52 \right) + C_T \|\rho(t_0)\|_{L^\infty}^{2+\a-\g} \left( \frac{1}{t_0^{\g/\a}} - \frac{1}{t_*^{\g/\a}} \right) \\
& <0.
\end{align*}
This contradicts the definition of $t_0$.  We conclude therefore that $t_0 = T$, so that the desired statement holds. 
\end{proof}	

Now that we have this bound on $\rho'$, it is easy to establish bounds on and $e'$.  To this end, we recall the relation 
\[
e' = \rho^2 \left( \frac{q'}{\rho} \right) + q \rho'.
\]
In light of \eqref{e:qxuprough} and our bounds on $\rho'$, we conclude that $e'$ is uniformly bounded on $\T\times [0,T]$ as well, by a constant $C_T$ which is allowed to depend on $\|f\|_{L^\infty_t W^{2,\infty}_x}$, $\cM$, $\a$, $T$, $\|\rho_0^{-1}\|_{L^\infty}$, and the $W^{1,\infty}$ norms of $u_0$, $\rho_0$, and $e_0$.

To bound $u'$, we need to consider two cases, depending on the values of $\a$.  When $\a\in (0,1)$, we simply recall that $u' = e + \L_\a \rho$, which we now know to be bounded on $\T\times [0,T]$ by a constant $C_T$ which is allowed to depend on $\|f\|_{L^\infty_t W^{2,\infty}_x}$, $\cM$, $\a$, $T$, $\|\rho_0^{-1}\|_{L^\infty}$, and the $W^{1,\infty}$ norms of $u_0$, $\rho_0$, and $e_0$.  For $\a\in [1,2)$, however, we do not yet have a bound on $\|\L_\a \rho\|_{\infty}$.  We argue as follows in the case $\a\in (1,2)$.  (Note that we do not include $\a=1$.) We differentiate \eqref{e:parabmo}, then replace $u'$ by $\frac{m'}{\rho} - \frac{u\rho'}{\rho}$ and evaluate at a maximum $x_+(t)$ of $|m'(\cdot, t)|$, to obtain
\begin{equation}
\p_t m' + \frac{|m'|^2}{\rho} - \frac{u\rho' m'}{\rho} + e'm + em' = -\rho' \L_\a m - \rho \L_\a m' + \rho' f + \rho f'
\end{equation}
(where we understand that all terms are evaluated at $(x_+(t), t)$.  Multiplying by $m'(x_+(t),t)$, we obtain (bracketing the lower-order terms)
\begin{align*}
\frac12 \p_t [(m')^2](x_+(t),t) 
& = - \frac{(m')^3}{\rho}(x_+(t),t) -\rho'm' \L_\a m(x_+(t),t) - \rho m' \L_\a m'(x_+(t),t) \\
& \quad \quad + \left[ \frac{u\rho' |m'|^2}{\rho} - e' m m' - e|m'|^2 + m' \rho' f + \rho m' f'\right](x_+(t),t),
\end{align*}
so that 
\begin{equation*}
\frac12 \p_t [(m')^2](x_+(t),t)
\le  C_T[\|m'(t)\|_{L^\infty}^3+1]  +  |\rho' m' \L_\a m(x_+(t),t)| - \rho m' \L_\a m'(x_+(t),t),
\end{equation*}
with $C_T$ depending only on the usual quantities.  To estimate $|\rho' m' \L_\a m(x_+(t),t)|$, we take $r=1$ in \eqref{e:Lbdinfty} to obtain 
\begin{align*}
|\rho' m'\L_\a m(x_+(t),t)| 
& \le C\|\rho'(t)\|_{L^\infty} \|m'(t)\|_{L^\infty} [D_\a m'(x_+(t),t)^{\frac12} + \|m\|_{L^\infty} ] \\ 
& \le \frac14\rho_-(t)D_\a m'(x_+(t),t) + C_T[\|m'\|_{L^\infty}^2 + 1].
\end{align*}
Applying \eqref{e:dissterm} and \eqref{e:nlmp} once more, we obtain the following bound: 
\begin{align*}
\frac12 \p_t [(m')^2](x_+(t),t) 
& \le  C_T[\|m'(t)\|_{L^\infty}^3+1]  +  |\rho' m' \L_\a m(x,t)| - \rho m'(x,t) \L_\a m'(x,t) \\
& \le C_T[\|m'(t)\|_{L^\infty}^3+1] - c_T\|m'(t)\|_{L^\infty}^{2+\a}.
\end{align*}
Since $\a>1$, we can conclude using similar reasoning as in the proof of the bound on $\|\rho'(t)\|_{L^\infty}$ (though the present situation is slightly simpler, since we do not have to reason differently for small and large times).  We now pause to record the obtained bounds as a Proposition:

\begin{PROP}
Let $(u, \rho)$ be a regular solution on the time interval $[0,\infty)$.  For each $T>0$, there exists a constant $C_T^{e'}$, which may depend on $\|f\|_{L^\infty_t W^{2,\infty}_x}$, $\cM$, $\a$, $T$, $\|\rho_0^{-1}\|_{L^\infty_x}$, and the $W^{1,\infty}$ norms of $u_0$, $\rho_0$, and $e_0$, such that 
\begin{equation}
\|e'(t)\|_{L^\infty} \le C_T^{e'},  \quad \quad t\in [0,T].
\end{equation}
If $\a\ne 1$, there exists a constant $C_T^{u'}$ depending on the same quantities, such that 
\begin{equation}
\|u'(t)\|_{L^\infty} \le C_T^{u'},  \quad \quad t\in [0,T].
\end{equation}
\end{PROP}

\subsection{Bounds in H\"older Spaces}
\label{ss:Holder2}
For $\a\ne 1$, we can now establish H\"older bounds on $u'$ and $\rho'$ in the same way that we treated $u$ and $\rho$ above.  We give the details only for $\rho'$. The right side of \eqref{e:rhox} is bounded in $L^\infty(0,T;L^\infty)$ for any $T>0$.  Therefore we can apply \cite{Silvestre} to \eqref{e:rhox} now, to conclude that for some $\g_1>0$, we have  
\begin{equation}
[\rho'(t)]_{C^{\g_1}} \le \frac{C_T}{t^{\g_1/\a}}\bigg( \|\rho'\|_{L^\infty(\T\times (0,T))} + \bigg\|\rho^3 \bigg( \frac{q'}{\rho} \bigg) + 2e\rho' + \rho' \L_\a \rho\bigg\|_{L^\infty(\T\times (0,T))}\bigg)\le \frac{C_T}{t^{\g_1/\a}}, 
\quad t\in (0,T].
\end{equation}
Here $C_T$ is allowed to depend on $\|f\|_{L^\infty_t W^{2,\infty}_x}$, $\cM$, $\a$, $T$, $\|\rho_0^{-1}\|_{L^\infty}$, and the $W^{1,\infty}$ norms of the initial data.  We record this, as well as the analogous bounds for $m'$ and $u'$, in the following proposition: 

\begin{PROP}
Let $(u, \rho)$ be a regular solution on the time interval $[0,\infty)$ and assume $\a\ne 1$.  Then for any $T>0$, there exists $\g_1>0$ such that $\rho$, $m=\rho u$, and $u$ satisfy bounds of the form 
\begin{equation}
\label{e:drhoHolder}
[\rho'(t)]_{C^{\g_1}}\le t^{-\g_1/\a} C_T, \quad t\in (0,T].
\end{equation}
\begin{equation}
\label{e:dmHolder}
[m'(t)]_{C^{\g_1}} \le t^{-\g_1/\a} C_T, \quad t\in (0,T];
\end{equation}
\begin{equation}
\label{e:duHolder}
[u'(t)]_{C^{\g_1}} \le t^{-\g_1/\a} C_T, \quad t\in (0,T];
\end{equation}
The constants $C_T$ may depend on $\|f\|_{L^\infty_t W^{2,\infty}_x}$, $\cM$, $\a$, $T$, $\|\rho_0^{-1}\|_{L^\infty}$, and the $W^{1,\infty}$ norms of $u_0$, $\rho_0$, and $e_0$.  The number $\g_1$ ultimately depends only on these same quantities.
\end{PROP}

\section{Strong Solutions}

\label{s:strong}

Our next goal is to prove the existence and uniqueness of global strong solutions.  We accomplish this in all cases under consideration except the case $\a=1$, where we prove only uniqueness.  We will proceed as follows.  First, we give the proof of existence for $\a\ne 1$; this follows essentially the same outline as the proof of the existence part of Theorem \ref{t:wk}.  Next, we take a brief detour to clarify the regularity of the time derivative $\p_t u$ in the case when $\a\in [1,2)$.  This discussion does not contain any deep facts, but it is strictly speaking necessary in order to carry out the integration-by-parts argument in our uniqueness argument.  As a byproduct of this discussion, though, we obtain a self-contained proof of the energy equality \eqref{e:ee} for strong solutions.  Finally, our proof of uniqueness follows a standard Gr\"onwall-type argument.  Note, however, that a bit of care is required in handling the dissipation term.  

\subsection{Existence ($\a\ne 1$)}
We have shown above that $u^\e$ and $\rho^\e$ are bounded sequences of $L^\infty(\d, T; C^{1,\g_1})$ for each $T>\d>0$. Applying the Aubin-Lions-Simon Theorem as before, we can find a subsequence of the sequence $\e_k$ constructed in the proof of Theorem \ref{t:wk}, which we will continue to label $\e_k$, such that $u^{\e_k}\to u$ and $\rho^{\e_k} \to \rho$ strongly in $C_\loc((0,\infty); C^1)$.  Now $(e^{\e_k})'$ is bounded in $L^\infty(\T\times [0,T])$; therefore it converges weak-$*$ (up to a subsequence) as $k\to \infty$ to some $g$ in the same class.  We claim that $g=e'$.  Indeed, we have as $k\to \infty$ that 
\[
\int_0^T \int_\T (e^{\e_k})' \f \dx\dt = \int_0^T \int_\T u^{\e_k} \f'' + \rho^{\e_k} \L_\a \f'\dx\dt
\to \int_0^T \int_\T u \f'' + \rho \L_\a \f'\dx\dt
= - \int_0^T \int_\T e \f' \dx\dt.
\]
But this limit is also equal to $\int_0^T \int_\T g\f$ by assumption.  Therefore $e$ is weakly differentiable in space, with weak derivative $e'=g$.  It follows that $(e^{\e_k})'$ converges weak-$*$ to $e'$.  Since $(u, \rho, e)$ is already known to be a weak solution, by Theorem \ref{t:wk}, we have now shown that $(u, \rho, e)$ is in fact a strong solution.

\subsection{Time Derivatives of Strong Solutions}
\label{ss:stee}

We begin by noting a few properties of strong solutions.  Most of these are basically obvious from the definition, but we believe it is useful to have them recorded explicitly.  First, we note that the evolution equations for $\rho$ and $e$ are true pointwise a.e., instead of merely in the weak sense; furthermore, all terms that appear in the equation belong to $L^\infty(0,T;L^\infty)$ for any $T>0$:
\begin{equation}
\rho_t + (\rho u)' = 0, \quad 
\text{ a.e., and in } L^\infty(0,T;L^\infty);
\end{equation}
\begin{equation}
e_t + (ue)' = f', \quad 
\text{ a.e., and in } L^\infty(0,T;L^\infty).
\end{equation}
The same is true for the $u$-equation if $\a<1$.  If $\a=1$, it may not be the case that $\L_\a(\rho u)\in L^\infty$, but it will belong to (for example) $L^2$.  (This is a rather academic point at the moment, though, since we have not proven the existence of strong solutions for $\a=1$.) If $\a > 1$, though, a pointwise a.e. interpretation is not available. However, we can still view the equation in $L^2 H^{-\a/2}$, as we did for weak solutions.
\begin{equation}
u_t + ue = -\L_\a(\rho u) + f, \quad 
\text{ a.e., and in } L^\infty(0,T;L^\infty), \text{ if } \a\in (0,1).
\end{equation}
\begin{equation}
u_t + ue = -\L_\a(\rho u) + f, \quad 
\text{ a.e., and in } L^2(0,T;L^2), \text{ if } \a\in (0,1].
\end{equation}
\begin{equation}
u_t + ue = -\L_\a(\rho u) + f, \quad 
\text{ in } L^2(0,T;H^{-\a/2}); \quad \a\in (0,2).
\end{equation}

Thus a bit of care is warranted in treating time derivatives of $u$.  We next show that the following formula is valid even when $\a\in (1,2)$:
\begin{equation}
\label{e:ptu2}
\frac12 \p_t u^2 =  u \p_t u = - u(ue) - u\L_\a(\rho u) + uf \quad 
\text{ in } L^2 H^{-\a/2}.
\end{equation}
Of course, when we write (for example) $u\p_t u$, we mean the element of $H^{-\a/2}$ defined by 
\[
\langle u\p_t u, \phi\rangle_{H^{-\a/2}, H^{\a/2}} = \langle \p_t u, u\phi\rangle_{H^{-\a/2}, H^{\a/2}};
\] 
the latter is perfectly well defined for any $\a\in (0,2)$ (but for different reasons, depending on whether $\a\in (0,1]$ or $\a\in (1,2)$). We will use this interpretation of $u\p_t u$ and similar elements of $H^{-\a/2}$ below without further comment.

Note that \eqref{e:ptu2} is obvious if $\a\in (0,1]$; therefore we assume below that $\a>1$.  But then $H^{\a/2}(\T)$ is an algebra, and $\|u\phi\|_{H^{\a/2}}\le C\|u\|_{H^{\a/2}} \|\phi\|_{H^{\a/2}}$.  Thus 
\[
|\langle u\L_\a (\rho u), \phi\rangle_{H^{-\a/2}\times H^{\a/2}}| = \int \L_{\a/2}(\rho u) \L_{\a/2}(u\phi)\dx \le C\|\rho u\|_{H^{\a/2}} \|u\|_{H^{\a/2}} \|\phi\|_{H^{\a/2}}.
\]
It follows that the right side of \eqref{e:ptu2} belongs to $L^2 H^{-\a/2}$ and is equal to $u\p_t u$ in this sense. It remains to show that this quantity is in fact equal to $\p_t u^2$.  To do this, we write 
\[
\langle u(s)^2, \phi\rangle_{H^{-\a/2}\times H^{\a/2}}\bigg|^t_0 
 = \int_\T u(s) \cdot u(s)\phi\dx\bigg|^t_0 
\]
for some time-independent function $\phi\in C^\infty(\T)$ and use the weak formulation of the $u$-equation, with $u\phi$ serving as the test function.  Note that since this weak formulation requires a very slight modification in this case to allow for the rough test function $u\phi$; to deal with this, we simply write a duality pairing in $H^{-\a/2}\times H^{\a/2}$ when necessary.  We have
\begin{align*}
\langle u(s)^2, \phi\rangle_{H^{-\a/2}\times H^{\a/2}}\bigg|^t_0 
& = \int_0^t \left( \langle \p_t u - \L_\a(\rho u), u\phi \rangle_{H^{-\a/2}\times H^{\a/2}} + \int [fu\phi - (ue)(u\phi)] \dx \right) \ds \\
& = \int_0^t \langle u\p_t u - u(ue) - u\L_\a(\rho u) + fu, \phi \rangle_{H^{-\a/2}\times H^{\a/2}} \ds \\
& = \int_0^t \langle 2u\p_t u,\phi\rangle_{H^{-\a/2}\times H^{\a/2}}\ds.
\end{align*}
This proves the claim.  Now we have $\rho, u^2\in L^2 H^1$, $\p_t \rho, \p_t(u^2)\in L^2 H^{-1}$, so we can apply the usual integration-by-parts formula to $\rho u^2$ as follows:
\begin{align*}
\frac12 \int \rho u^2(s)\dx \bigg|^t_0
& = \frac12 \int_0^t \langle \p_t \rho, u^2\rangle_{H^{-1}\times H^1} + \langle u \p_t u, \rho\rangle_{H^{-1}\times H^1}\ds \\
& = \int_0^t \int (\rho u) (uu')\dx + \langle -ue -\L_\a(\rho u) + f, \rho u\rangle_{H^{-\a/2}\times H^{\a/2}}\ds \\
& = \int_0^t \rho u^2(u'-e) + \rho u f - \langle \L_\a(\rho u), \rho u\rangle_{H^{-\a/2}\times H^{\a/2}}.
\end{align*}
Then recalling the definitions of $e$ and $\cT$, this yields
\begin{align*}
\frac12 \int \rho u^2(s)\dx \bigg|^t_0
& = \int_0^t \int \rho u f\dx + \langle \cT(\rho, u), \rho u\rangle_{H^{-\a/2}\times H^{\a/2}}\ds \\
& = -\int_0^t \int_\R \int_\T \rho(x) \rho(y) \frac{|u(x)-u(y)|^2}{|x-y|^{1+\a}}\dx\dy\ds + \int_0^t \int_\T \rho u f\dx\ds.
\end{align*}
Thus we have a self-contained proof of the validity of the energy equality \eqref{e:ee} for strong solutions.  One can prove \eqref{e:rhoee} even more easily.  The main point here, though, is the validity of the integration by parts.  We will use this below in our proof of the uniqueness.  

\subsection{Uniqueness}

Next, we prove uniqueness. Let $(u_1, \rho_1, e_1)$ and $(u_2, \rho_2, e_2)$ be two solutions to \eqref{e:mainv}--\eqref{e:maind} with the same initial data.  We assume $u_i, \rho_i, e_i\in W^{1,\infty}$ for $i=1,2$. Define
\[
\begin{array}{ccccccc}
u_\d = u_1 - u_2, & & \rho_\d = \rho_1 - \rho_2, & & e_\d = e_1 - e_2, & & q_\d = q_1 - q_2, \\
u_\s = u_1 + u_2, & & \rho_\s = \rho_1 + \rho_2, & & e_\s = e_1 + e_2, & & q_\s = q_1 + q_2.
\end{array}
\]
Then by the integration-by-parts formula (which is valid for all $\a\in (0,2)$ by the discussion in the previous subsection), we have 
\begin{align*}
\int \rho_\s u_\d^2(s)\dx \bigg|^t_0
& = \int_0^t \langle \p_t \rho_\s, u_\d^2\rangle_{H^{-1}\times H^1} + 2\langle u_\d \p_t u_\d, \rho_\s\rangle_{H^{-1}\times H^1}\ds \\
& = 2\int_0^t \int (\rho u)_\s (u_\d u_\d')\dx + \langle -(ue)_\d -\L_\a(\rho u)_\d, \rho_\s u_\d \rangle_{H^{-\a/2}\times H^{\a/2}}\ds.
\end{align*}
Expanding $(\rho u)_\s$, $(ue)_\d$ and $(\rho u)_\d$ yields
\begin{align*}
\int \rho_\s u_\d^2(s)\dx \bigg|^t_0 & = \int_0^t \int \rho_\s u_\s u_\d (u_\d'-e_\d) + \rho_\d u_\d^2 u_\d' - \rho_\s u_\d^2 u_\s' + \rho_\s u_\d^2 \L_\a \rho_\s \dx\ds \\
& \quad - \int_0^t \langle \L_\a(\rho_\s u_\d) + \L_\a(\rho_\d u_\s), \rho_\s  u_\d \rangle_{H^{-\a/2}\times H^{\a/2}}\ds.
\end{align*}
Using the definitions of $e_\d$ and $\cT$, then symmetrizing, we obtain
\begin{align*}
\int \rho_\s u_\d^2(s)\dx \bigg|^t_0 
& = \int_0^t \int [\rho_\d u_\d' - \rho_\s u_\s']u_\d^2\dx + \langle \cT(\rho_\s, u_\d) + \cT(\rho_\d, u_\s), \rho_\s u_\d \rangle_{H^{-\a/2}\times H^{\a/2}}\ds \\
& = \int_0^t \int [\rho_\d u_\d'- \rho_\s u_\s']u_\d^2\dx\ds -\frac12 \int_0^t \int_\R \int_\T \rho_\s(x)\rho_\s(y) \frac{|u_\d(x) -u_\d(y)|^2}{|x-y|^{1+\a}}\dx\dy\ds \\
& \quad + \int_0^t \langle \cT(\rho_\d, u_\s), \rho_\s u_\d \rangle_{H^{-\a/2}\times H^{\a/2}}\ds.
\end{align*}
For convenience, let us denote 
\[
\varepsilon_\d(t):=\frac12 \int_0^t \int_\R \int_\T \rho_\s(x)\rho_\s(y) \frac{|u_\d(x) -u_\d(y)|^2}{|x-y|^{1+\a}}\dx\dy\ds.
\]
We can estimate the term $\int |\langle \cT(\rho_\d, u_\s), \rho_\s  u_\d \rangle_{H^{-\a/2}\times H^{\a/2}}|\dx$ as follows.
\begin{align*}
& \int |\L_{\a/2}(\rho_\s u_\d u_\s)||\L_{\a/2}(\rho_\d)| + |\L_{\a/2}(\rho_\s u_\d)||\L_{\a/2}(\rho_\d u_\s)|\dx \\
& \quad \quad \le C\|\rho_\s u_\s u_\d\|_{H^{\a/2}} \|\rho_\d\|_{H^{\a/2}} +  C\|\rho_\s u_\d\|_{H^{\a/2}} \|\rho_\d u_\s\|_{H^{\a/2}} \\
& \quad \quad \le C\|\rho_\s u_\s\|_{W^{1,\infty}} \|u_\d\|_{H^{\a/2}} \|\rho_\d\|_{H^{\a/2}} +  C\|\rho_\s\|_{W^{1,\infty}} \|u_\d\|_{H^{\a/2}} \|\rho_\d\|_{H^{\a/2}} \|u_\s\|_{W^{1,\infty}} \\
& \quad \quad \le C^*\|\rho_\d\|_{H^{\a/2}}^2 + \varepsilon_\d(t).
\end{align*}
Above, we have repeatedly used the fact that $\|fg\|_{H^{a/2}}\le C\|f\|_{W^{1,\infty}} \|g\|_{H^{\a/2}}$ for $f\in W^{1,\infty}$, $g\in H^{\a/2}$.  To see why this inequality holds, simply note that 
\begin{align*}
\|fg\|_{H^{\a/2}}^2 
& = \int_\R \int_\T \frac{|fg(x)-fg(y)|^2}{|x-y|^{1+\a}}\dx\dy \\
& \le \int_\R \int_\T \frac{2 |f(x)|^2 |g(x)-g(y)|^2}{|x-y|^{1+\a}}\dx\dy + \int_\R \int_\T \frac{2|f(x)-f(y)|^2|g(y)|^2}{|x-y|^{1+\a}}\dx\dy \\ 
& \le 2\|f\|_{L^\infty}^2 \|g\|_{H^{\a/2}}^2 + C\|f\|_{W^{1,\infty}}^2 \|g\|_{L^2}^2 \le C\|f\|_{W^{1,\infty}}^2 \|g\|_{H^{\a/2}}^2.
\end{align*}

Thus, for any $\a\in (0,2)$, we have
\begin{equation}
\label{e:u2dtag1}
\| \sqrt{\rho_\s} u_\d(t)\|_{L^2}^2 \le C\int_0^t \|u_\d(s)\|_{L^2}^2\ds + C^*\int_0^t \|\rho_\d(s)\|_{H^{\a/2}}^2\ds.
\end{equation}
At this stage in the proof, the (possibly quite large) constant $C^*$ may appear worrisome.  But as we will see below, the $\rho_\d$ equation carries a term of the form $-\int_0^t \|\rho_\d(s)\|_{H^{\a/2}}^2\ds$.  Therefore, by multiplying the entire $\rho_\d$ equation by $C^*$ and adding the result to \eqref{e:u2dtag1}, we may absorb this bad term.  

We now treat the $\rho_\d$ equation.  We obtain 
\begin{align*}
\frac{\mathrm{d}}{\dt}\int \frac{\rho_\d^2}{\rho_\s}\dx 
& = \int \frac{2\rho_\d \p_t \rho_\d}{\rho_\s} + \rho_\d^2 \p_t\left( \frac{1}{\rho_\s} \right) \dx \\
& \le -\int \frac{\rho_\d (\rho_\s' u_\d + \rho_\s u_\d' +\rho_\d' u_\s + \rho_\d u_\s') }{\rho_\s}\dx + C\|\rho_\d\|_{L^2}^2.
\end{align*}
Using the fact that 
\[
u_\d' = \frac12(q_\d \rho_\s + q_\s \rho_\d ) + \L_\a \rho_\d,
\]
we write 
\begin{align*}
\int \rho_\d u_\d'\dx 
& = \frac12 \int (q_\d \rho_\s + q_\s \rho_\d)\rho_\d \dx + \int |\L_{\a/2} \rho_\d|^2 \dx.
\end{align*}
Therefore 
\begin{align*}
\frac{\mathrm{d}}{\dt} \left\| \frac{\rho_\d}{\sqrt{\rho_\s}} \right\|_{L^2}^2 
& \le -\int \frac{\rho_\d (\rho_\s' u_\d + \rho_\s u_\d' +\rho_\d' u_\s + \rho_\d u_\s') }{\rho_\s}\dx + C\|\rho_\d\|_{L^2}^2\\
&  = -\int \rho_\d u_\d \frac{\rho_\s'}{\rho_\s}\dx - \frac12\int (q_\d \rho_\s + q_\s \rho_\d)\rho_\d\dx - \|\L_{\a/2}\rho_\d\|_{L^2}^2 \\
& \quad \quad + \frac12 \int \rho_\d^2 \frac{(\rho_\s u_\s)'}{\rho_\s^2}\dx + C\|\rho_\d\|_{L^2}^2\\
& \le C[\|u_\d\|_{L^2}^2 + \|\rho_\d\|_{L^2}^2 + \|q_\d\|_{L^2}^2] - \|\L_{\a/2}\rho_\d\|_{L^2}^2.
\end{align*}
So adding up, integrating in time, and multiplying by the constant $C^*$ from \eqref{e:u2dtag1}, we obtain
\begin{equation}
\label{e:rho2dt}
C^*\left\| \frac{\rho_\d}{\sqrt{\rho_\s}}(t) \right\|_{L^2}^2
\le C\int_0^t \|u_\d(s)\|_{L^2}^2 + \|\rho_\d(s)\|_{L^2}^2 + \|q_\d(s)\|_{L^2}^2\ds - C^*\int_0^t \|\rho_\d(s)\|_{H^{\a/2}}^2\ds.
\end{equation}
Adding this to \eqref{e:u2dtag1}, we obtain
\begin{equation}
\label{e:urho2dt}
\|\sqrt{\rho_\s} u_\d(t)\|_{L^2}^2 + C^* \left\| \frac{\rho_\d}{\sqrt{\rho_\s}}(t)\right\|_{L^2}^2 
\le C\int_0^t \|u_\d(s)\|_{L^2}^2 + \|\rho_\d(s)\|_{L^2}^2 + \|q_\d(s)\|_{L^2}^2\ds.
\end{equation}
	
Finally, we deal with the $q_\d$ equation. 
\begin{equation}
\label{e:q2dt}
\frac{\mathrm{d}}{\dt}\int q_\d^2\dx = - \int u_\d q_\s' q_\d + u_\d q_\d' q_\d + \frac{2f'}{\rho_1 \rho_2}\rho_\d q_\d \dx \le C\|u_\d\|_{L^2}^2 + C\|\rho_\d\|_{L^2}^2 + C\|q_\d\|_{L^2}^2. 
\end{equation}

Integrating and adding to \eqref{e:urho2dt}, we obtain 	
\begin{equation}
\label{e:urhoq2dt}
\|\sqrt{\rho_\s} u_\d(t)\|_{L^2}^2 + C^* \left\|\frac{\rho_\d}{\sqrt{\rho_\s}}(t)\right\|_{L^2}^2 + \|q_\d(t)\|_{L^2}^2 
\le C\int_0^t \|u_\d(s)\|_{L^2}^2 + \|\rho_\d(s)\|_{L^2}^2 + \|q_\d(s)\|_{L^2}^2\ds.
\end{equation}

This proves that $u_\d$, $\rho_\d$, and $q_\d$ are identically zero, thus establishing uniqueness.

\subsection{The Case of a Compactly Supported Force $(\a\ne 1)$}

We finally note that, for $\a\ne 1$, the construction above gives our solution sufficient regularity so that we can prove that flocking occurs in the special case where $f\equiv 0$ (or when $f$ is compactly supported in time).  The key observation is that the velocity field $u$ is $C^1$ for all positive time; therefore we can apply the results of \cite{STII} (when $\a>1$) and \cite{STIII} (when $\a<1$) to show the existence of a flocking pair.  (Actually, the results of \cite{STII} are stated and proved only for the case $\a=1$, but trivial adjustments give the analogous statements and proofs for $\a>1$.)  We quote the results of intermediate steps without proof, giving details only for the existence of a flocking state.  

First, membership of $u(t)$ in $C^1(\T)$, $t>0$ allows us to prove the estimate
\begin{equation}
\label{e:dissenh}
D_\a u'(x) \ge \frac{c|u'(x)|^{2+\a}}{A(t)^\a},
\end{equation}
where $A(t)$ denotes the diameter of the velocities, as before, and $c$ is some positive absolute constant.  Recall that in Section \ref{ss:0forcewk} we showed that $A(t)$ decays exponentially quickly in time if $f$ is compactly supported.  Thus \eqref{e:dissenh} allows us to absorb powers of $u'$ into $D_\a u'$ after a finite time.  For a proof of \eqref{e:dissenh}, see Lemma 3.3 in \cite{STII} and Lemma 2.2 in \cite{STIII}.  The estimate \eqref{e:dissenh} is used to prove that
\begin{equation}
\label{e:duflock}
\|u'(\cdot, t)\|_{L^\infty}\le Ce^{-\delta t},
\end{equation}
for some $\d>0$; see Lemma 3.4 in \cite{STII} and Lemma 2.3 in \cite{STIII}.  Thus the convergence of $u(t)$ toward the constant $\overline{u}$ (c.f. Section \ref{ss:Holder1}) occurs exponentially quickly in $W^{1,\infty}$ rather than just in $L^\infty$.  

Now, in the case where $f$ is compactly supported in time, all of the bounds of Section \ref{s:Wbds} can be taken to be constant bounds. Therefore, for large $t$ (say $t>1$), $u(t)$ and  $\rho(t)$ are uniformly bounded in $C^{1,\g_1}$ for some $\g_1>0$, and $\rho(t)$ is bounded in $W^{1,\infty}$ for \textit{all} $t$.  Defining $\widetilde{\rho}(x,t) = \rho(x - t\overline{u}, t)$ and writing the density equation in the moving reference frame, we conclude that $\|\widetilde{\rho}_t\|_{L^\infty}\le C e^{-\d t}$ on account of the uniform boundedness of $\rho(t)$ in $W^{1,\infty}$ and the estimate \eqref{e:duflock}.  It follows that $\widetilde{\rho}$ converges in $L^\infty$ to a limiting state $\rho_\infty$.  Defining $\overline{\rho}(x,t) = \rho_{\infty}(x-t\overline{u})$, this completes the proof of fast flocking in $W^{1,\infty}\times L^\infty$ toward $(\overline{u}, \overline{\rho})$.  

We can conclude flocking in stronger spaces if we are willing to sacrifice the exponential rate of convergence.  Since $u(t)$ and $\rho(t)$ are uniformly bounded in $C^{1,\g_1}$, we must in fact have convergence in $C^{1,\e}\times C^{1,\e}$ for every $\e\in (0,\g_1)$ by compactness and by uniqueness of the limit in $W^{1,\infty}\times L^\infty$.

\section{Acknowledgments}
T.M. Leslie wishes to thank his advisor Roman Shvydkoy for stimulating discussions and valuable advice.  The work of T.M. Leslie is supported in part by NSF Grant DMS-1515705 (PI: Roman Shvydkoy).


\end{document}